%% file: CDC-2022.tex
\documentclass[letterpaper, 10 pt, conference]{ieeeconf}
\IEEEoverridecommandlockouts
\usepackage[noadjust]{cite}
\usepackage{amsmath,amssymb,amsfonts}

\usepackage{graphicx}
\usepackage{textcomp}
\usepackage{xcolor}

\usepackage[utf8]{inputenc}
\usepackage{mathrsfs}
\usepackage{stmaryrd}
\usepackage{graphicx}
\usepackage{mathtools}
\usepackage{tikz}

\usepackage{array}

\usepackage[normalem]{ulem}

\usepackage[linesnumbered,ruled,vlined]{algorithm2e}
\usepackage{algorithmic}

\DeclarePairedDelimiter{\norm}{\lVert}{\rVert} % Norm
\DeclarePairedDelimiter{\abs}{\lvert}{\rvert} % Absolute value

\DeclareMathOperator{\diff}{d\!}
\DeclareMathOperator{\degree}{deg}

\newcommand{\suchthat}{\ifnum\currentgrouptype=16 \mathrel{}\middle|\mathrel{}\else\mid\fi}

\def \R {{\mathbb R}}

\SetCommentSty{mycommfont}
\SetKwInput{KwInput}{Input}
\SetKwInput{KwOutput}{Output} 

\algsetup{linenosize=\small}

\newcommand{\EXP} [1] {\textup{e}^{#1}}
\newcommand{\I} [0] {\dot{\iota}}

\newtheorem{theorem}{Theorem}
\newtheorem{lemma}[theorem]{Lemma}
\newtheorem{proposition}[theorem]{Proposition}
\newtheorem{remark}[theorem]{Remark}

\begin{document}

\title{\LARGE \bf MID Property for Delay Systems: Insights on Spectral Values with Intermediate Multiplicity}

\author{Islam Boussaada$^{\ast \dagger}$ \and Guilherme Mazanti$^{\ast}$ \and Silviu-Iulian Niculescu$^{\ast}$ \and Amina Benarab$^{\ast \dagger}$% <-this % stops a space
\thanks{$^{\ast}$The authors are with Universit\'e Paris-Saclay, CNRS, CentraleSup\'elec, Inria, Laboratoire des Signaux et Syst\`emes, 91190, Gif-sur-Yvette, France; Corresponding author: \texttt{guilherme.mazanti@inria.fr}}%
\thanks{$^{\dagger}$IB \& AB are also with IPSA Paris, 94200 Ivry sur Seine, France.}%
}

\maketitle

\begin{abstract}   
This paper focuses on the problem of \emph{multiplicity induced dominancy} (MID) for a class of linear time-invariant systems represented by delay-differential equations. If the problem of generic MID was characterized in terms of properties of the roots of Kummer hypergeometric functions, the case of \emph{intermediate} MID is still an open problem. The aim of this paper is to address such a problem by using the Green--Hille transformation for characterizing the distribution of the nonasymptotic zeros of linear combinations of Kummer functions. An illustrative example completes the presentation and shows the effectiveness of the proposed methodology.
\end{abstract}

\section{Introduction}

A common feature in modeling transport and propagation phenomena, signal transmission in communication networks, or age structure in population dynamics is their time heterogeneity, and one of the simplest way to represent such processes and/or phenomena is by using \emph{delays\/} in their mathematical models. For further examples, we refer to \cite{Stepan1989Retarded,mcdo:89-bio,Michiels2014Stability,Insperger2011Semi,els:73,kol-nosov:86} and the references therein. As pointed out in \cite{Sipahi2011Stability}, the presence of delay in the system's dynamics may have a dichotomic effect and a lot of methods and techniques have been proposed in the open literature to address these stability issues (see, e.g., \cite{Avellar1980Zeros, Cruz1970Stability, Diekmann1995Delay, Gu2003Stability, Hale1993Introduction, Hale2003Stability, Hayes, Henry1974Linear, Michiels2014Stability,Pinney1958Ordinary}).  

In the linear time-invariant (LTI) systems whose dynamics are represented by delay-differential equations (DDEs), we have a particular interesting property, called \emph{multiplicity induced dominancy (MID)\/} that, to the best of the authors' knowledge, was not sufficiently addressed in the open literature. More precisely, the MID property simply says that the characteristic root of maximal multiplicity is the rightmost root of the spectrum, i.e., all other roots are located to its left in the complex plane. In other words, this characteristic root of maximal multiplicity is nothing else than the \emph{spectral abscissa\/}\footnote{A deeper discussion on spectral abscissa can be found in \cite{Michiels2014Stability}.} of the dynamical system. As pointed out in \cite{MBN-2021-JDE,BMN-2022-CRAS}, this property opens an interesting perspective in control through the so-called \emph{partial pole placement\/} method with guarantees on the spectrum location of the remaining characteristic roots. Further discussions on existing methods for characterizing multiple roots can be found in \cite{NBLMM21-tds}.

The case of the maximal allowable multiplicity\footnote{i.e., the quasipolynomial degree} of a characteristic root, called \emph{generic MID property\/}, was addressed and completely characterized in \cite{MBN-2021-JDE} (retarded case) and in \cite{BMN-2022-CRAS} (unifying retarded and neutral cases) for LTI DDEs including a \emph{single delay\/} in their models. The proposed arguments to prove such a property are based on some analytical properties of Kummer and Whittaker confluent hypergeometric functions, which turn out to be essential for the factorization of the corresponding characteristic functions. It should be noted that such a method cannot be extended straightforwardly to the case of \emph{spectral values with intermediate multiplicity}, a fact that represents a drawback of the method. However, as shown in \cite{balogh2021conditions} by using different arguments that exploit the structure of the system, the MID property holds also in some non-generic cases, i.e., for spectral values with intermediate multiplicities. In tuning PID controllers, a different argument was proposed in \cite{MBCBNC-Automatica-2022} for some class of unstable systems. Unfortunately, extending the methods in \cite{balogh2021conditions, MBCBNC-Automatica-2022} to a more general setting is not a trivial task. Finally, for similar problems, different parameter-based methods have been discussed in \cite{VMG13-cst,RMGS15-tac} when using the delays as control parameters. To summarize, to the best of the authors' knowledge, there does not exist any systematic approach to prove dominancy of the characteristic roots with intermediate multiplicity. 

The aim of this paper is to address such a problem and to outline the ideas of a new method that could also encompass intermediate multiplicities. More precisely, the contribution of the paper is threefold. First, we provide conditions under which spectral values with intermediate multiplicity are dominant. Second, the method used to obtain such conditions also has an interest on itself: we make use of the Green--Hille (integral) transformation introduced by Hille one century ago \cite{hille1922} for characterizing the location of the nonasymptotic zeros of Whittaker hypergeometric functions. It should be noted that these ideas complete the previous approaches based on the properties of Kummer hypergeometric functions to handle \emph{generic MID\/} in the retarded and neutral cases (see, e.g., \cite{MBN-2021-JDE,BMN-2022-CRAS}). To the best of the authors' knowledge, such a method represents a novelty in the open literature. Finally, as a byproduct of the analysis, new insights in MID control of the dynamics of a pendulum are proposed. In fact, we explore some of the existing links between the intermediate MID and the spectral abscisa optimization problem.

The remaining of the paper is organized as follows. Spectral bounds for retarded systems, a motivating example, prerequisites, and the formulation of the problem considered in the paper are presented in Section~\ref{sec:prob-MID}. The main results are derived in Section~\ref{sec:MID-int-mult}, where we first establish second-order equations for a linear combination of two Kummer functions and for a corresponding Whittaker-type function, before proving the MID property for an intermediate multiplicity and providing discussions on frequency bounds in the right-half plane for the spectra of dynamical systems represented by DDEs. An illustrative example is discussed in Section~\ref{sec:exm-MID-int-mult}, and some concluding remarks in Section~\ref{sec:conclusion} end the paper.

\emph{Notations.} Throughout the paper, the following notations are used: $\mathbb N^\ast$, $\mathbb R$, $\mathbb C$ denote the sets of positive integers, real numbers, and complex numbers, respectively, and we set $\mathbb N = \mathbb N^\ast \cup \{0\}$. The set of all integers is denoted by $\mathbb Z$ and, for $a, b \in \mathbb R$, we denote $\llbracket a, b\rrbracket = [a, b] \cap \mathbb Z$, with the convention that $[a, b] = \emptyset$ if $a > b$. For a complex number $\lambda$, $\Re(\lambda)$ and $\Im(\lambda)$ denote its real and imaginary parts, respectively. The open left and right complex half-planes are the sets $\mathbb C_-$ and $\mathbb C_+$, respectively, defined by $\mathbb C_- = \{\lambda \in \mathbb C \suchthat \Re(\lambda) < 0\}$ and $\mathbb C_+ = \{\lambda \in \mathbb C \suchthat \Re(\lambda  ) > 0\}$. For $\alpha \in \mathbb C$ and $k \in \mathbb N$, $(\alpha)_k$ is the \emph{Pochhammer symbol} for the \emph{ascending factorial}, defined inductively as $(\alpha)_0 = 1$ and $(\alpha)_{k+1} = (\alpha+k) (\alpha)_k$.

\section{Prerequisites and Problem Formulation}
\label{sec:prob-MID}

Consider the LTI dynamical system described by the DDE
\begin{equation}
\label{eq:DDE}
y^{(n)}(t) + \sum_{k=0}^{n-1} a_k y^{(k)}(t) + \sum_{k=0}^m \alpha_k y^{(k)}(t - \tau) = 0,
\end{equation}
under appropriate initial conditions, where $y(\cdot)$ is the real-valued unknown function, $\tau > 0$ is the delay, and $a_0,\allowbreak \dotsc,\allowbreak a_{n-1},\allowbreak \alpha_0,\allowbreak \dotsc,\allowbreak \alpha_{m}$ are real coefficients.
The DDE \eqref{eq:DDE} is said to be of \emph{retarded type}\footnote{in the case when the highest order of derivation appears only in the non-delayed term $y^{(n)}(t)$.} if $m < n$, or of \emph{neutral type} if $m = n$. We refer to \cite{Hale1993Introduction, Michiels2014Stability} for a deeper discussions on DDEs and related results and properties.

Notice that \eqref{eq:DDE} is a particular case of the time-delay system
\begin{equation}\label{eq:DDEM}
\dot \xi(t)+B_\tau\dot \xi(t-\tau)=A_{0}\xi(t)+A_{\tau} \xi(t-\tau),
\end{equation}
where $\xi(t) = (y(t),\,y'(t),\,\ldots,y^{(n-1)}(t))^{T}\in \mathbb{R}^n$ is the state vector and $A_0,\,A_\tau,\,B_\tau\in \mathcal{M}_n(\mathbb{R})$ are real-valued matrices which can be easily constructed from \eqref{eq:DDE}.

Consider a positive integer $n_p\in \mathbb{N}^*$, and an open set $\mathcal{O}\in\mathbb{R}^{n_p}$. For a set of parameters $\Vec p\in\mathcal{O}$, assume that the coefficients of the DDEs  \eqref{eq:DDE} $a_k$ ($k\in\llbracket 0, n-1\rrbracket$) and $\alpha_k$ ($k\in\llbracket 0, m\rrbracket$) are sufficiently smooth functions depending on the parameters $\Vec p$. Assume further that the delay $\tau:\mathcal{O} \to \mathbb{R}_+$ is a sufficiently smooth, nonnegative, and bounded function for all the parameters $ \Vec{p}\in\mathcal{O}$. Then the characteristic function associated with \eqref{eq:DDE} is the quasipolynomial $\Delta: \mathbb{C} \times \mathcal{O} \to \mathbb{C}$ defined by
\begin{equation}
\label{eq:Delta-P0-Ptau}
\Delta(\lambda, \Vec p)=P_0(\lambda, \Vec p)+P_{\tau}(\lambda, \Vec p)
\EXP{-\lambda\tau(\Vec p)},
\end{equation}
where $P_0$ and $P_{\tau}$ are the polynomials with real coefficients given by
\begin{equation}
\label{eq:quasi-PQ}
P_0(\lambda, \Vec p)=\lambda^n+\sum_{k=0}^{n-1} a_k(\Vec p)\lambda^k,\quad P_{\tau}(\lambda)=\sum_{k=0}^{m} \alpha_k(\Vec p)\lambda^k.
\end{equation}
Such a parameter vector $\Vec p$ may collect the delay $\tau$ and all $n+m+1$ coefficients of the DDE \eqref{eq:DDE} or it may define some particular structure of the coefficients of the polynomials $P_0$ and $P_\tau$, some particular dependence between the coefficients of the polynomials and the delay, or it may reflect the way the controller's gains appear in the characteristic function of the closed-loop system. 
For the sake of simplicity, in most of the cases, if no ambiguity, we will simply use $\Delta(\lambda)$, $P_0(\lambda)$, and $P_\tau(\lambda)$. 

It is well-known that the exponential stability of the trivial solution of \eqref{eq:DDE} can be described by the location of the (infinitely many) \emph{characteristic roots\/} of $\Delta$ (see, e.g., \cite{BC63,Michiels2014Stability}).

The \emph{degree} of the quasipolynomial $\Delta$ from \eqref{eq:Delta-P0-Ptau} is the integer $\degree(\Delta) = n + m + 1$. As discussed in \cite{BN-ACAP-2016}, this integer, which is larger than the degrees of the polynomials $P_0$ ($\degree(P_0)=n$) and $P_\tau$ ($\degree(P_\tau)=m$), is nothing else than the integer appearing in the P\'olya--Szeg\H{o} bound from \cite[Part Three, Problem~206.2]{Polya1998Problems}, and also corresponds to the maximal allowable multiplicity that a characteristic root of \eqref{eq:Delta-P0-Ptau}--\eqref{eq:quasi-PQ} may have. In addition, a characteristic root reaching such a bound is necessarily real.

\begin{remark}
\label{remk:imaginary}
On the imaginary axis, the characteristic roots of the quasipolynomial $\Delta$ defined by \eqref{eq:Delta-P0-Ptau} admit a bounded frequency, i.e., a bounded imaginary part. Indeed, any imaginary root $\lambda_0=\I\,\omega_0$ of $\Delta$  necessarily  satisfies
\begin{equation*}
    \abs{P_0(\I\,\omega_0)}^2=\abs{P_{\tau}(\I\,\omega_0)}^2.
\end{equation*}
The function $\mathcal{F}$ defined by $\mathcal{F}(\omega)=\abs{P_0(\I\,\omega)}^2-\abs{P_{\tau}(\I\,\omega)}^2$ is a polynomial on $\omega$ with real coefficients, and thus all its positive roots can be bounded in terms of its coefficients (see, for instance, \cite{Marden:49}). However, such an observation does not provide insights on frequency bounds for other roots, in particular roots on $\mathbb{C}_+$.
\end{remark}

\subsection{Spectrum Distribution for Retarded Delay Systems}
Despite the fact that the characteristic function of some DDE has an infinite number of characteristic roots, retarded systems, that is \eqref{eq:DDE} with $m<n$ or, equivalently, \eqref{eq:DDEM} with $B_\tau=0$, admit finitely many roots on any vertical strip in the complex plane \cite[Chapter~1, Lemma~4.1]{Hale1993Introduction}. 

Several general results on the location of roots of \eqref{eq:Delta-P0-Ptau} can be found in the literature and, in particular, we refer the interested reader to \cite{Boese-98} for generic result on the location of associated spectral values for arbitrary $n$. The next proposition collects two interesting properties, whose proofs can be found, respectively, in \cite{Michiels2014Stability} and \cite{partington2004h}.

\begin{proposition}\label{ENVELOPE}
Consider the LTI system \eqref{eq:DDE}, the corresponding system \eqref{eq:DDEM}, and their characteristic quasipolynomial $\Delta$ given by \eqref{eq:Delta-P0-Ptau}--\eqref{eq:quasi-PQ}.
\begin{enumerate}
    \item If $m<n$ and $\lambda$ is a characteristic root of system \eqref{eq:DDEM} with $B_\tau = 0$, then it satisfies
 \begin{equation}\label{EN1}
\abs{\lambda} \leq \norm{A_{0}+A_{\tau}\,\EXP{-\tau \lambda}}_{2}.
\end{equation}
 \item If $m=n$ and $\displaystyle \lim_{\abs{\lambda} \to \infty} \abs{P_\tau(\lambda)/P_0(\lambda)} <1$, then the characteristic equation $\Delta$ defined by \eqref{eq:quasi-PQ} has a finite number of unstable roots in the right half-plane.
 \end{enumerate}
\end{proposition}

\begin{remark}
Inequality \eqref{EN1}, combined with the triangular inequality, provides a generic \emph{envelope curve} around the characteristic roots corresponding to system \eqref{eq:DDEM}. In other words, the equality case in \eqref{EN1} defines a curve in the complex plane such that all characteristic roots of $\Delta$ are located to the left of that curve. We refer to \cite{Mori1989} for further insights on spectral envelopes for retarded time-delay systems with a single delay.
\end{remark}

\subsection{Motivating Example: Controlling the Inverted Pendulum}

Consider now a dynamical system modeling a friction-free inverted pendulum on a cart. The model adopted here was discussed in \cite{Krauskopf2004, S2005, BMN2015} and, in the sequel, we keep the same notations.
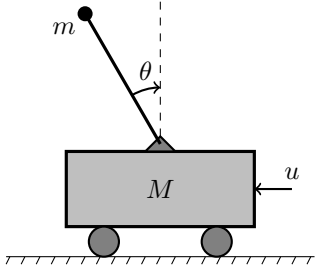
\begin{figure}[!ht]
\begin{center}
\input{Figures/invertedPendulum.tex}
\caption{Inverted pendulum on a cart.}\label{IP}
\end{center}
\end{figure}
In the dimensionless form, the dynamics of the inverted pendulum on a cart in Fig.~\ref{IP} is governed by the second-order differential equation
\begin{equation}\label{InvertedPendulum}
\left(1 - \frac{3\epsilon}{4}\cos^{2}(\theta)\right)\ddot\theta+\frac{3\epsilon}{8}\dot\theta^2\sin(2\theta)-\sin\theta + u\cos\theta = 0,
\end{equation}
where $\epsilon={m}/{(m+M)}$, $M$ is the mass of the cart, $m$ is the mass of the pendulum, and $u$ represents the control law, which is the horizontal driving force. If $\epsilon \neq \frac{4}{3}$, then the linearized system around the equilibrium point $\theta = \dot\theta = u = 0$ is
$\ddot \theta + \frac{u - \theta}{1 - \frac{3 \epsilon}{4}} = 0$.

We assume that the system is controlled by using a standard delayed PD controller of the form $u(t)=k_p\,\theta(t-\tau)+k_d\,\dot\theta(t-\tau)$, with $(k_p,k_d)\in\mathbb{R}^2$. The local stability of the closed-loop system is then reduced to the study of the location of the spectrum of the quasipolynomial $\Delta(\lambda, k_p, k_d, \tau) = \lambda^2 - \frac{1}{1 - \frac{3 \epsilon}{4}} + \frac{\EXP{-\lambda \tau}}{1 - \frac{3 \epsilon}{4}} \left(k_d \lambda + k_p\right)$ as a function of the system's parameters $(k_p,k_d,\tau)$. A generalized Bogdanov--Takens singularity with codimension three is identified in \cite{Krauskopf2004}. It should be mentioned that $\degree(\Delta)=4$ and that the system free of delay is of second-order. In this case, the multiplicity $3$ represents an intermediate multiplicity larger than the degree of the system free of delays.

\subsection{Whittaker Functions and Hille Oscillation Theorem}
\label{sec:Whittaker-Hille}

We shall use in this paper some classical hypergoemetric functions, which we present now. The first such function we introduce is the \emph{Kummer confluent hypergeometric function}, which, for $a, b \in \mathbb C$ such that $-b \notin \mathbb N$, is the entire function
$\Phi(a, b, \cdot): \mathbb C \to \mathbb C$ 
defined by the series
\begin{equation}
\label{DefiConfluent}
\Phi(a, b, z) = \sum_{k=0}^{\infty} \frac{(a)_k}{(b)_k} \frac{z^k}{k!}.
\end{equation}
The series in \eqref{DefiConfluent} converges for every $z \in \mathbb C$ and, as presented in \cite{Buchholz1969Confluent, Erdelyi1981Higher, Olver2010NIST}, it satisfies the \emph{Kummer differential equation}
\begin{equation}
\label{KummerODE}
z \frac{\partial^2 \Phi}{\partial z^2}(a, b, z) + (b - z) \frac{\partial \Phi}{\partial z}(a, b, z) - a \Phi(a, b, z) = 0.
\end{equation}
As discussed in \cite{Buchholz1969Confluent, Erdelyi1981Higher, Olver2010NIST}, for every $a, b, z \in \mathbb C$ such that $\Re(b) > \Re(a) > 0$, Kummer functions also admit the integral representation
\begin{equation}
\label{eq:integral-Kummer}
\Phi(a, b, z) = \frac{\Gamma(b)}{\Gamma(a) \Gamma(b - a)} \int_0^1 \EXP{zt} t^{a-1} (1-t)^{b-a-1} \diff t,
\end{equation}
 where $\Gamma$ denotes the Gamma function. This integral representation has been exploited in \cite{MBN-2021-JDE} to characterize the spectrum of some DDEs.
 
Kummer functions satisfy some recurrence relations, often called  \emph{contiguous relations}, see for instance \cite{Olver2010NIST}. In particular, the following relations are of interest.

\begin{lemma}[{\cite[p.~325]{Olver2010NIST}}]
\label{lemma:contiguous}
Let $a, b, z \in \mathbb C$ with $a \neq b$, $z\neq0$, and $-b \notin \mathbb N$. The following re\-la\-tions hold:
\begin{equation}\label{Contiguous}
\begin{aligned}
&\Phi(a , b +1, z)= 
\frac{-b \left(a +z \right)\, \Phi(a , b , z)+a b\, \Phi(a +1, b , z)}{z \left(a -b \right)},\\
&\Phi(a +1, b +1, z)= 
-\frac{-b\, \Phi(a +1, b , z)+b\, \Phi(a , b , z)}{z}.
\end{aligned}
\end{equation}
\end{lemma}

Kummer confluent hypergeometric functions have close links with \emph{Whittaker functions}. For $k, l \in \mathbb C$ with $-2l \notin \mathbb N^\ast$, the \emph{Whittaker function} $\mathcal M_{k, l}$ is the function defined for $z \in \mathbb C$ by
\begin{equation}
\label{KummerWhittaker}
\mathcal{M}_{k,l}(z) = \EXP{-\frac{z}{2}} z^{\tfrac{1}{2} + l} \Phi(\tfrac{1}{2} + l - k, 1 + 2 l, z),
\end{equation}
(see, e.g., \cite{Olver2010NIST}). Note that, if $\frac{1}{2} + l$ is not an integer, the function $\mathcal M_{k, l}$ is a multi-valued complex function with branch point at $z = 0$. The nontrivial roots of $\mathcal M_{k, l}$ coincide with those of $\Phi(\tfrac{1}{2} + l - k, 1 + 2 l, \cdot)$ and $\mathcal M_{k, l}$ satisfies the \emph{Whittaker differential equation}
\begin{equation}
\label{Whittaker}
\varphi''(z) =\left(\frac{1}{4}-\frac{k}{z}+\frac{l^2-\frac{1}{4}}{z^2}\right)\varphi(z).
\end{equation}
Since $\mathcal M_{k, l}$ is a nontrivial solution of the second-order linear differential equation \eqref{Whittaker}, any nontrivial root of $\mathcal M_{k, l}$ is necessarily simple.

In \cite{hille1922}, Hille studies the distribution of zeros of functions of a complex variable satisfying linear second-order homogeneous differential equations with variable coefficients, as is the case for the degenerate Whittaker function $\mathcal{M}_{k,l}$, which satisfies \eqref{Whittaker}. Thanks to an integral transformation defined there and called \emph{Green--Hille transformation}, and some further conditions on the behavior of the function, Hille showed how to discard regions in the complex plane from including complex roots.

Consider, for instance, the general homogeneous second-order differential equation
\begin{equation}
\label{eq:Hille}
\frac{d}{dz}\left[K(z)\,\frac{d\,\varphi}{dz}(z)\right]+G(z)\varphi(z)=0,
\end{equation}
where $z$ is complex and the functions $G$ and $K$ are assumed analytic in some region $\Theta$ such that $K$ does not vanish in that region. Equation~\eqref{eq:Hille} can be written in $\Theta$ as a second-order system on the unknown functions $\varphi_1(z)=\varphi(z)$ and $\varphi_2(z)=K(z)\,\frac{d\,\varphi}{dz}(z)$, and the Green--Hille transformation consists on multiplying the equation on $\varphi_1$ by $\overline{\varphi_2(z)}$, that on $\varphi_2$ by $\overline{\varphi_1(z)}$, and integrating on $z$ along a path in $\Theta$, which yields
\begin{multline}\label{Green}
\left[\overline{\varphi_1(z)}\,\varphi_2(z)\right]_{z_1}^{z_2}-\int_{z_1}^{z_2}\abs{\varphi_2(z)}^2\frac{\overline{\diff z}}{\mbox{ }\overline{K(z)}\mbox{ }} \\ {} + \int_{z_1}^{z_2}\abs{\varphi_1(z)}^2 G(z) \diff z = 0,
\end{multline}
where $z_1, z_2 \in \Theta$ and both integrals are taken along the same arbitrary smooth path in $\Theta$ connecting $z_1$ to $z_2$.

The following result, which is proved in \cite{BMN-2022-BSM} using the Green--Hille transformation from \cite{hille1922}, gives insights on the distribution of the nonasymptotic zeros of Kummer hypergeometric functions with real arguments $a$ and $b$. 

\begin{proposition}[\cite{BMN-2022-BSM}]
\label{CorZerosKummer}
Let $a,\,b\in \mathbb R$ be such that $b\geq 2$.
\begin{enumerate}
\item\label{CorZerosKummer-k-eq-0} If $b = 2a$, then all nontrivial roots $z$ of $\Phi(a,b, \cdot)$ are purely imaginary. 
\item\label{CorZerosKummer-k-geq-0} If $b > 2a$ (resp., $b < 2a$), then all nontrivial roots $z$ of $\Phi(a,b, \cdot)$ satisfy $\Re(z) > 0$ (resp., $\Re(z) < 0$).
\item\label{CorZerosKummer-k-neq-0} If $b \neq 2a$, then all nontrivial roots $z$ of $\Phi(a,b, \cdot)$ satisfy
\[(b - 2a)^2 {\Im(z)}^2 - \left(4a(b-a) - 2b\right) {\Re(z)}^{2} > 0.\]
\end{enumerate}
\end{proposition}

\begin{remark}
\label{remk:combination}
In feedback control theory, one of the major interests of Proposition~\ref{CorZerosKummer} is the fact that a quasipolynomial admitting a characteristic root of maximal multiplicity, equal to its degree, can be factorized in terms of a Kummer function. As discussed in the sequel, a quasipolynomial with a root with intermediate multiplicity also shares its remaining roots with an appropriate linear combination of Kummer functions. Unfortunately, to the  best of the authors' knowledge, there does not exist any result in the open literature describing the distribution of the nonasymptotic zeros of such function combinations.
\end{remark}

\subsection{Problem Formulation}

Consider now the DDE \eqref{eq:DDE} and its characteristic function $\Delta$ given by \eqref{eq:Delta-P0-Ptau}--\eqref{eq:quasi-PQ}. As indicated in \cite{BN-ACAP-2016}, $\degree(\Delta)=n+m+1$. 

We say that a characteristic root $\lambda_0$ of $\Delta$ satisfies the \emph{MID property\/} if (i) its \emph{algebraic multiplicity\/} (denoted by $M(\lambda_0)$) is \emph{larger than one}, and (ii) it is \emph{dominant\/} in the sense that all the characteristic roots $\lambda_\sigma$ of the spectrum satisfy the condition $\Re(\lambda_\sigma)\leq \Re(\lambda_0)$. In other words, $\lambda_0$ is the rightmost root of the spectrum and defines the \emph{spectral abscissa\/} of the quasipolynomial $\Delta$.
In the case $M(\lambda_0)=\degree(\Delta)$, it was shown in \cite{MBN-2021-JDE} (case $m = n - 1$)  and \cite{BMN-2022-CRAS} (general case   $m \leq n$) that $\lambda_0$ satisfies the MID property. This ``limit'' case is also called \emph{generic MID}. 

The problem addressed in this paper can be formulated as follows: \emph{finding conditions on the parameters of the dynamical system \eqref{eq:DDE} such that a characteristic root $\lambda_0$ with \emph{intermediate\/} algebraic multiplicity $M(\lambda_0)$ verifying $n+1\leq M(\lambda_0) \leq n+m$ satisfies the MID property.\/} More precisely, and for the sake of brevity, our focus will be to derive \emph{appropriate conditions guaranteeing that $M(\lambda_0)=n+m$}. Such an intermediate MID leaves one degree of freedom in terms of system's parameters\footnote{For instance, the delay or a feedback gain may appear as being appropriate.}. In control, such a parameter may be used to improve the performances of the corresponding closed-loop system. The general case $n+1\leq M(\lambda_0) \leq n+m$ can still be addressed by similar arguments, but Lemma~\ref{lemma:edo-F} below will involve linear combination of more Kummer functions, the expression \eqref{eq:Delta-max-mult} for $\Delta(\lambda)$ in Theorem~\ref{thm:max-mult} will involve an integral containing a more general polynomial in $t$ than $1 - A t$, and the parameter vector $\Vec{p}$ in Theorem~\ref{thm:MID} will depend on the coefficients of such a polynomial.

The PD control of the inverted pendulum in the case of delay in the input/output channel considered as a motivating example corresponds to such a situation. More precisely, in that case, we have $\degree(\Delta) = 4$, $n = 2$, and hence the only possible intermediate multiplicity is $M(\lambda_0)=3$.

\section{Main Results}
\label{sec:MID-int-mult}

\subsection{Some Insights on Linear Combinations of Two Kummer Functions} 
\label{sec:generalized-K-W}

As discussed in Remark~\ref{remk:combination}, beyond the standard contiguous relation, to the best of the authors' knowledge, there does not exist any result describing the distribution of the nonasymptotic zeros of linear combinations of Kummer functions. 

The next lemma provides a partial step towards that goal, by providing a non-autonomous second-order differential equation having a given linear combination of Kummer functions as a solution.

\begin{lemma}
\label{lemma:edo-F}
Let $a,\,b$ be two complex numbers and $\alpha$ and $\beta$ two real numbers and define the parameter vector  $\Vec p=(a,\,b,\,\alpha,\,\beta)$. Then the complex function $F$ defined by
\begin{equation}
\label{eq:generalized-Kummer}
F(z,\Vec p) = \alpha\,\Phi(a, b, z) + \beta\, \Phi(a, b+1, z),
\end{equation}
with $z\notin\{0,\,\frac{\beta  \left(\beta +\alpha \right) b^{2}}{\left(\left(a -b \right) \alpha -\beta  b \right) \alpha}\}$, satisfies the second-order differential equation
\begin{equation}
\label{eq:edo-F}
\frac{\partial^2 F}{\partial z^2}(z,\Vec p)  + Q(z,\Vec p) \frac{\partial F}{\partial z}(z,\Vec p) + R(z,\Vec p) F(z,\Vec p) = 0,
\end{equation}
where
\begin{align}
Q(z,\Vec p) & = -1+\frac{b +1}{z}-\frac{\alpha  \left(a \alpha -\alpha  b -\beta  b \right)}{D(z,\Vec p)}, \label{eq:def-Q} \displaybreak[0] \\
R(z,\Vec p) & = -\frac{N(z,\Vec p)}{D(z,\Vec p)}, \label{eq:def-R} \displaybreak[0] \\
\intertext{with}
N(z,\Vec p) & = a \left(\left(\left(a -b \right) \alpha^{2}-\alpha  b \beta \right) z -\beta  b \left(b +1\right) \alpha \right) \notag \\
& \hphantom{=} -a\,b^{2} \beta^{2}, \notag \displaybreak[0] \\
D(z,\Vec p) & = \left(\left(a -b \right) \alpha^{2}-\alpha  b \beta \right) z -\alpha  \,b^{2} \beta -b^{2} \beta^{2}. \notag
\end{align}
\end{lemma}

Lemma~\ref{lemma:edo-F} can be proved by using that $\frac{\partial \Phi}{\partial z} (a, b, z) = \frac{a}{b} \Phi(a+1, b+1, z)$, which follows immediately from \eqref{DefiConfluent}, and exploiting the contiguous relations from Lemma~\ref{lemma:contiguous}. In the sequel, we shall refer to functions $F$ of the form \eqref{eq:generalized-Kummer} as \emph{Kummer-type functions}.

Note that Whittaker functions are defined in terms of Kummer functions in \eqref{KummerWhittaker} by using the multiplicative factor $\EXP{-\frac{z}{2}} z^{\frac{1}{2} + l}$, thanks to which the Whittaker differential equation \eqref{Whittaker} has no first-order term. We now proceed similarly from Kummer-type functions in order to define \emph{Whittaker-type functions}. The next lemma can be shown by straightforward computations.

\begin{lemma}
Let $a, b$ be two complex numbers, $\alpha, \beta$ be two real numbers, $F$ be the function defined in \eqref{eq:generalized-Kummer}, and $Q$ and $R$ be given by \eqref{eq:def-Q} and \eqref{eq:def-R}, respectively. 

Let $\mathcal Q$ be a primitive of $\frac{Q}{2}$ and define the function $W$ by
\begin{equation}
\label{eq:generalized-Whittaker}
W(z,\Vec p) = \EXP{\mathcal Q(z, \Vec p)} F(z, \Vec p).
\end{equation}
Then $W$ satisfies the second-order differential equation
\begin{equation}
\label{eq:Whittaker-type-edo}
\frac{\partial^2 W}{\partial z^2}(z,\Vec p) + G(z,\Vec p) W(z,\Vec p) = 0,
\end{equation}
where
\begin{equation}
\label{eq:def-G}
G(z,\Vec p) = R(z,\Vec p) - \frac{(Q(z,\Vec p))^2}{4} - \frac{1}{2}\frac{\partial Q}{\partial z}(z,\Vec p).
\end{equation}
\end{lemma}

In the sequel, we refer to functions $W$ of the form \eqref{eq:generalized-Whittaker} as \emph{Whittaker-type functions}.

\subsection{Necessary and Sufficient Conditions for the Intermediate Multiplicity $M(\lambda_0) = n + m$}

Thanks to the preliminary results of Section~\ref{sec:generalized-K-W}, we are now in position to prove the following result, providing a necessary and sufficient condition for a given real number $\lambda_0$ to be a root of a quasipolynomial $\Delta$ with multiplicity $n + m$.

\begin{theorem}
\label{thm:max-mult}
Let $\tau > 0$, $\lambda_0 \in \mathbb R$, and consider the quasipolynomial $\Delta$ from \eqref{eq:Delta-P0-Ptau}--\eqref{eq:quasi-PQ}. The number $\lambda_0$ is a root of multiplicity at least $n + m$ of $\Delta$ if and only if there exists $A \in \mathbb R$ such that
\begin{multline}
\label{eq:Delta-max-mult}
\Delta(\lambda) = \frac{\tau^{m} (\lambda - \lambda_0)^{n + m}}{(m-1)!} \\ \cdot \int_0^1 t^{m-1} (1 - t)^{n - 1} (1 - A t) \EXP{-t \tau(\lambda - \lambda_0)} \diff t.
\end{multline}
\end{theorem}

\begin{proof}
Let $\mathcal V$ be the set of all functions $\Delta$ of the form $\Delta(\lambda) = P_0(\lambda) + \EXP{-\lambda \tau} P_\tau(\lambda)$ with $P_0$ and $P_\tau$ given by \eqref{eq:quasi-PQ}. Note that $\mathcal V$ is an affine subspace of the space of all entire complex functions with $\dim \mathcal V = n + m + 1$. Let us denote by $\mathcal V_{\lambda_0}^{n + m}$ the subset of $\mathcal V$ of those functions $\Delta$ admitting $\lambda_0$ as a root of multiplicity at least $n + m$, i.e.,
\begin{multline*}
\mathcal V_{\lambda_0}^{n + m} = \{\Delta \in \mathcal V \suchthat \Delta^{(k)}(\lambda_0) = 0 \text{ for all } \\ k \in \{0, \dotsc, n + m - 1\}\}.
\end{multline*}
All equations $\Delta^{(k)}(\lambda_0) = 0$, $k \in \{0, \dotsc, n + m - 1\}$, are linearly independent, and thus $\mathcal V_{\lambda_0}^{n + m}$ is a subspace of $\mathcal V$ of codimension $n + m$, i.e., $\dim \mathcal V_{\lambda_0}^{n + m} = 1$.

Introduce now $\mathcal W_{\lambda_0}^{n, m}$ as the space of all functions $\Delta$ of the form \eqref{eq:Delta-max-mult} for some $A \in \mathbb R$. Clearly, $\mathcal W_{\lambda_0}^{n, m}$ is an affine subspace of the space of all entire complex functions with $\dim \mathcal W_{\lambda_0}^{n, m} = 1$.

As a first step, we will prove that $\mathcal W_{\lambda_0}^{n, m} \subset \mathcal V$, i.e., that every function $\Delta$ of the form \eqref{eq:Delta-max-mult} is indeed a quasipolynomial of the form \eqref{eq:Delta-P0-Ptau}--\eqref{eq:quasi-PQ}. To do so, we first notice that, by an immediate inductive integration by parts, we have (see also \cite[Proposition~2.1]{MBN-2021-JDE})
\begin{equation}
\label{eq:integral-p-exp}
\int_0^1 p(t) \EXP{-z t} \diff t = \sum_{k = 0}^d \frac{p^{(k)}(0) - p^{(k)}(1) \EXP{-z}}{z^{k+1}}
\end{equation}
for every $z \in \mathbb C \setminus \{0\}$, $d \in \mathbb N$, and $p$ a polynomial of degree $d$. Now, let $\Delta \in \mathcal W_{\lambda_0}^{n, m}$ and $A \in \mathbb R$ be such that $\Delta$ is given by \eqref{eq:Delta-max-mult}. By using \eqref{eq:integral-p-exp}, we deduce that
\begin{multline*}
\Delta(\lambda) = \sum_{k = 0}^{n + m - 1} \tau^{k-n} (\lambda - \lambda_0)^k p^{(n + m - k - 1)}(0) \\ - \EXP{-\tau (\lambda - \lambda_0)} \sum_{k = 0}^{n + m - 1} \tau^{k-n} (\lambda - \lambda_0)^k p^{(n + m - k - 1)}(1),
\end{multline*}
where $p$ is the polynomial $p(t) = t^{m-1} (1-t)^{n-1} (1 - A t)$. In particular, we have $p(0) = p'(0) = \dotsb = p^{(m-2)}(0) = 0$, $p(1) = p'(1) = \dotsb = p^{(n-2)}(1) = 0$. In addition, we have $p^{(m-1)}(0) = 1$. Hence
\begin{multline*}
\Delta(\lambda) =  (\lambda - \lambda_0)^n + \sum_{k = 0}^{n-1} \tau^{k-n} (\lambda - \lambda_0)^k p^{(n + m - k - 1)}(0) \\ - \EXP{-\tau (\lambda - \lambda_0)} \sum_{k = 0}^{m} \tau^{k-n} (\lambda - \lambda_0)^k p^{(n + m - k - 1)}(1),
\end{multline*}
and thus, as required, $\Delta \in \mathcal V$.

We now notice that $\mathcal W_{\lambda_0}^{n, m} \subset \mathcal V_{\lambda_0}^{n + m}$, since, for any $\Delta$ given by \eqref{eq:Delta-max-mult}, $\lambda_0$ is clearly a root of multiplicity at least $n + m$ of $\Delta$. Since $\mathcal W_{\lambda_0}^{n, m}$ and $\mathcal V_{\lambda_0}^{n + m}$ are both affine spaces of dimension $1$, we conclude that $\mathcal W_{\lambda_0}^{n, m} = \mathcal V_{\lambda_0}^{n + m}$, yielding the conclusion.
\end{proof}

\subsection{MID Validity for the Intermediate Multiplicity $M(\lambda_0)=n+m$}

Finally, from Theorem~\ref{thm:max-mult}, we are able to provide some (appropriate) \emph{sufficient conditions\/} under which we have the MID property for characteristic roots of multiplicity $n + m$ of $\Delta$.

\begin{theorem}
\label{thm:MID}
Let $\tau > 0$, $\lambda_0$ and $A$ be real numbers, and $\Delta$ be given by \eqref{eq:Delta-max-mult}. Consider the parameter vector
\[\Vec p = \left(m, n + m, \frac{(1-A) (n-1)!}{(n + m - 1)!}, \frac{A n!}{(n + m)!}\right)\]
and let $F$ and $G$ be defined respectively as in \eqref{eq:generalized-Kummer} and  \eqref{eq:def-G}. Assume that, for every $t \in (0, 1)$ and every root $z$ of $F(\cdot, \Vec p)$ in $\mathbb C_-$, we have $\Re[z G(t z, \Vec p)] \geq 0$. Then $\lambda_0$ is a dominant root of $\Delta$, i.e., $\lambda_0$ satisfies the MID property.
\end{theorem}

\begin{proof}
By using the trivial identity $1 - A t = 1 - A + A(1 - t)$, one obtains from \eqref{eq:integral-Kummer}, \eqref{eq:generalized-Kummer}, and \eqref{eq:Delta-max-mult} that 
\begin{equation}
\label{eq:Delta-F}
\Delta(\lambda) = \tau^{m} (\lambda - \lambda_0)^{n + m} F(-\tau(\lambda - \lambda_0), \Vec p).
\end{equation}
In particular, all roots of $\Delta$ different from $\lambda_0$ are roots of $\lambda \mapsto F(-\tau(\lambda - \lambda_0), \Vec p)$, and the result is thus proved if one shows that all roots of the Kummer-type function $F(\cdot, \Vec p)$ have nonnegative real part.

To do so, we consider the Whittaker-type function $W(\cdot, \Vec p)$ from \eqref{eq:generalized-Whittaker}. Applying Hille's method to \eqref{eq:Whittaker-type-edo}, by taking in \eqref{eq:Hille} $z_1 = 0$ and $z_2$ equal to a root $z_\ast$ of $F(\cdot, \Vec p)$, we obtain:
\[
\int_0^{z_\ast} \abs*{W'(z)}^2 \overline{\diff z} = \int_0^{z_\ast} \abs*{W(z)}^2 G(z) \diff z, 
\]
where, for the sake of simplicity, we omit the dependence of $W$ and $G$ on $\Vec p$. We choose as integration path the line segment from $0$ to $z_\ast$. Hence
\[
\overline{z_\ast} \int_0^1 \abs*{W'(t z_\ast)}^2 \diff t = z_\ast \int_0^1 \abs*{W(t z_\ast)}^2 G(t z_\ast) \diff t.
\]
Taking the real part, we get
\begin{equation}
\label{eq:proof-Hille}
x_\ast \int_0^1 \abs*{W'(t z_\ast)}^2 \diff t = \int_0^1 \abs*{W(t z_\ast)}^2 \Re\left[z_\ast G(t z_\ast)\right] \diff t,
\end{equation}
where $x_\ast = \Re(z_\ast)$ and $y_\ast = \Im(z_\ast)$.

Assume now, by contradiction, that $F(\cdot, \Vec p)$ admits a root with negative real part, and take $z_\ast$ in \eqref{eq:proof-Hille} as equal to this root. The left-hand side of \eqref{eq:proof-Hille} is negative, however its right-hand side is nonnegative by assumption, yielding the desired contradiction. Hence all roots of $F(\cdot, \Vec p)$ have nonnegative real parts, yielding the conclusion.
\end{proof}

\subsection{DDEs Frequency Bound in the Right Half-Plane}

The main difficulty when applying Theorem~\ref{thm:MID} is to verify the technical assumption $\Re[z G(t z, \Vec p)] \geq 0$ for every $t \in (0, 1)$ and every root $z$ of $F(\cdot, \Vec p)$ in $\mathbb C_-$ or, equivalently, to verify that $\Re[z G(- t z, \Vec p)] \leq 0$ for every $t \in (0, 1)$ and every root $z$ of $z \mapsto \Delta(\lambda_0 + \frac{z}{\tau})$ in $\mathbb C_+$. For that purpose, a useful technique is to establish a priori information on the location of roots of $\Delta$ with real part greater than $\lambda_0$, and in particular bounds on their imaginary parts.

To do so, a standard first step is to introduce the normalized quasipolynomial $\tilde\Delta(z) = \tau^n \Delta(\lambda_0 + \frac{z}{\tau})$, which can be written as $\tilde\Delta(z) = \tilde P_0(z) + \EXP{-z} \tilde P_\tau(z)$ for some suitable polynomials $\tilde P_0$ and $\tilde P_\tau$ of degrees $n$ and $m$, respectively. Hence, the problem of studying eventual roots of $\Delta$ with real part greater than $\lambda_0$ reduces to the study of eventual roots of $\tilde\Delta$ with positive real part.

A possible strategy to do so is to follow ideas similar to those of Remark~\ref{remk:imaginary}, i.e., to notice that any root $z$ of $\tilde\Delta$ satisfies
\[\abs{\tilde{P_0} (x + \I \omega)}^2 \EXP{2x} = \abs{\tilde{P_\tau}(x + \I \omega)}^2,\]
where $x = \Re(z)$ and $\omega = \Im(z)$. In particular, if $z$ has nonnegative real part, then $\EXP{2x} \geq T_\ell(x)$, where, for $\ell \in \mathbb N$, the polynomial $T_\ell$ is the truncation of the Taylor expansion of $\EXP{2x}$ at order $\ell$, i.e., $T_\ell(x) = \sum_{k = 0}^\ell \frac{(2x)^\ell}{\ell !}$. Hence, any root $z = x + \I \omega$ of $\tilde\Delta$ with nonnegative real part satisfies
\[\mathcal F(x, \omega) \geq 0,\]
where $\mathcal F$ is the polynomial given by
\[\mathcal F(x, \omega) = \abs{\tilde{P_\tau}(x + \I \omega)}^2 - \abs{\tilde{P_0} (x + \I \omega)}^2 T_\ell(x).\]
In addition, $\mathcal F$ only depends on $\omega$ through $\omega^2$ (which is a consequence of the fact that $\tilde P_0$ and $\tilde P_\tau$ are polynomials with real coefficients), and one may thus introduce the variable $\Omega = \omega^2$ and define the polynomial $H$ by setting $H(x, \Omega) = F(x, \sqrt{\Omega})$ for $\Omega \geq 0$. Hence, any root $z = x + \I \omega$ of $\tilde\Delta$ with nonnegative real part satisfies
\begin{equation}
\label{eq:bound-imag-H}
H(x, \Omega) \geq 0,
\end{equation}
where $\Omega = \omega^2$. One can thus establish a bound on the imaginary parts of roots of $\tilde\Delta$ by exploiting the polynomial inequality \eqref{eq:bound-imag-H}. This has been done for some low-order cases in \cite{benarab2022multiplicity, Benarab2020MID}. In particular, all these works have shown that it is sufficient to bound the absolute value of the imaginary parts of the roots in the right half-plane by $\pi$, as one can in general easily exclude by other arguments, such as those from Theorem~\ref{thm:MID}, the possibility of having roots in the right-half plane with imaginary part at most $\pi$, thus concluding the proof of dominance of $\lambda_0$.

The procedure described in this subsection is synthetized in Algorithm~\ref{algo:s} (see \cite{benarab2022multiplicity}), in which one increases the order of the Taylor expansion of $\EXP{2x}$ until a suitable bound is found.

\begin{algorithm}[ht] 
\algsetup{linenosize=\tiny} 
\KwInput{$\tilde{\Delta}(z)= \tilde{P_0}(z)+ \tilde{P_\tau}(z)\, \EXP{-z};$ \tcp{Normalized quasipolynomial}}
\KwInput{\texttt{maxOrd};  \tcp{Maximal order}}

\tcp{Initialization}

$\texttt{ord}=0$;
\tcp{\texttt{ord}: order of truncation of the Taylor expansion of $\EXP{2\,x}$;}

$\texttt{dominance} = \texttt{false}$;

\While{(not \texttt{dominance}) and ($\texttt{ord} \leq \texttt{maxOrd}$)}{
    Set $\mathcal{F}(x,\omega) = \abs{\tilde{P_\tau} (x+\I \omega)}^2 - \abs{\tilde{P_0} (x+\I \omega)}^2  T_{\texttt{ord}}(x)$;
    
    \tcp{$T_{\texttt{ord}}(x)$: Taylor expansion of $\EXP{2x}$ of $\text{order} = \texttt{ord}$}
    
    Set $H(x,\Omega) = \mathcal{F}(x, \sqrt{\Omega})$;
    \tcp{$H$ is a polynomial}
    
    Set $\Omega_{k}(x)$ as the $k$-th real root of $H(x, \cdot)$;
    
    \If{$\displaystyle \sup_{x \geq 0} \max_{k} \Omega_k(x) \leq \pi^2$}{
        $\texttt{dominance} = \texttt{true}$;
    }
    
    $\texttt{ord} = \texttt{ord}+1$;
}

\KwOutput{Frequency bound: If \texttt{dominance} is \texttt{true}, then $\abs{\omega} \leq \pi$ for every root of $\tilde\Delta$ with positive real part;}
\caption{Estimation of a frequency bound for time-delay differential equations with a single delay}
\label{algo:s}
\end{algorithm}

\section{An Illustrative Example: GMID, Intermediate MID, and Pendulum Stabilization}
\label{sec:exm-MID-int-mult}

Consider the dynamical system modeling a friction-free classical pendulum \cite{A1999}, whose dynamics are are governed by the following second-order differential equation:
\begin{equation}\label{InvPenNLin}
\ddot \theta(t) + \frac{g}{L}\sin(\theta(t)) = u(t),
\end{equation}
where $\theta(t)$ denotes the angular displacement of the pendulum at time $t$ with respect to the stable equilibrium position, $L$ is the pendulum length, $g$ is the gravitational acceleration, and $u(t)$ is the control input, which stems from an applied external torque.
Assume that the control law is a standard delayed PD controller of the form
\begin{equation}\label{DPD}
u(t)=-k_p\,\theta(t-\tau)-k_d\,\dot\theta(t-\tau),
\end{equation}
with $(k_p,k_d)\in\mathbb{R}^2$. The local stability of the closed-loop system is reduced to the study of the location of the roots of the quasipolynomial 
\begin{equation}
\label{DeltaPendulum}
   \Delta(\lambda)=\lambda^{2}+\frac{g}{L}+\left(k_{d} \lambda +k_{p}\right) {e}^{-\lambda\tau}.
\end{equation}
It is easy to see that $\degree(\Delta)=4$. By applying the GMID property, it follows that the only admissible quadruple root is $\lambda_0 = -\sqrt{2 g / L}$ and it is achieved if the system's parameters $(k_p,k_d,\tau)$ verify $k_{d} = -{{e}^{-2} \sqrt{2 g / L}},\quad k_{p} = -{5 \,{e}^{-2} g}/{L},\quad\tau = \sqrt{2 L/{g}}$.

As precised in \cite{BMN-2022-CRAS}, the GMID does not allow any degree of freedom in assigning $\lambda_0$. In order to allow for some additional freedom when assigning $\lambda_0$, one can relax such a constraint by forcing the root $\lambda_0$ to have a multiplicity lower than the maximal, and consider, for instance, the delay as a \emph{free tuning parameter}. This motivates the study of a (non-generic) MID property, which was carried out, for instance, in \cite{Boussaada2020Multiplicity} for second-order systems, see  also \cite{BMN-2022-CRAS}.

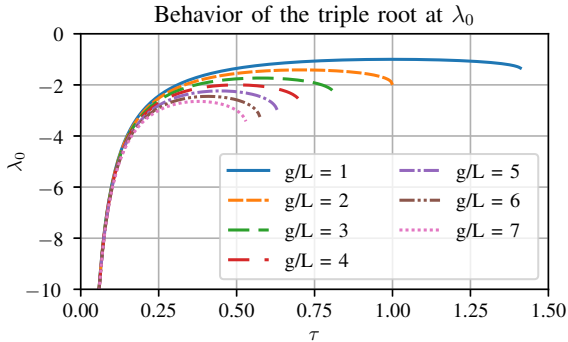
\begin{figure}[ht]
\centering
\resizebox{0.9\columnwidth}{!}{\input{Figures/MID_pendulum_3.pgf}}
\vspace*{-1em}
\caption{The behavior of the triple root (spectral abscissa) of \eqref{DeltaPendulum} at $\lambda=\lambda_0$ given by \eqref{lambda0} as a function of the free delay parameter $0<\tau<\sqrt{2\,L/g}$ for $g/L \in \{1,\,\dotsc,\,7\}$. Clearly, increasing the ratio $g/L$ decreases the assignment region as well as the delay margin.}
\label{Pendulum-TRIPLE}
\end{figure} 

\begin{proposition} \label{m3}
For any $0<\tau < \sqrt{{2\,L}/{g}}$, let
\begin{equation}\label{lambda0}
    \lambda_0=\frac{-2+\sqrt{-\frac{g \,\tau^{2}}{L}+2}}{\tau}.
\end{equation}
The delayed PD controller \eqref{DPD} with
\begin{equation}
\label{eq:kd-kp}
 \begin{aligned}
k_{d} = 
\frac{2 \left(\tau  \lambda_{0}+1\right) {\mathrm e}^{\tau  \lambda_{0}}}{\tau},\; 
k_{p} = 
\frac{2 \left(5 L \tau  \lambda_{0}+g \,\tau^{2}+3 L \right) {\mathrm e}^{\tau  \lambda_{0}}}{\tau^{2} L}
 \end{aligned}   
 \end{equation}
and $\lambda_0\,\tau\geq-1$, locally exponentially stabilizes the system \eqref{InvPenNLin}. Furthermore, the intermediate MID property holds with an exponential decay rate $\lambda_0$ for the closed-loop system.
\end{proposition}

\begin{proof}
The normalization $\tilde\Delta(z) = \tau^2 \Delta(\lambda_0 + \frac{z}{\tau})$ of $\Delta$ from \eqref{DeltaPendulum} with \eqref{lambda0}--\eqref{eq:kd-kp} is
\begin{equation*}
    \tilde{\Delta}(z)=\left(  \left( 2\,\mu+2 \right) z+4\,\mu+2 \right) {{\rm e}^{-z}}+{z}^{2}+2\,
\mu\,z-4\,\mu-2,
\end{equation*}
where $\mu=\tau\,\lambda_0$. The integral representation \eqref{eq:Delta-max-mult} of $\tilde\Delta$ is
\begin{equation*}
\tilde{\Delta}(z) = {z}^{3}\int_{0}^{1}\! q_\mu (t)\, {{\rm e}^{-zt}}\diff t,
\end{equation*}
where $q_\mu(t) = \left( -1-2\,\mu \right) {t}^{2}+2\,\mu\,t+1$, and it can be further written as a combination of two Kummer functions as
\[\tilde{\Delta}(z)=z^3\,[ 2\left( 1+ \mu \right) {\Phi\left(1,\,3,\,-z\right)} - \left(1+2 \,\mu \right) {\Phi\left(1,\,4,\,-z\right)}].\]
Note that $q_\mu$ keeps a constant sign for $t \in (0, 1)$ if and only if $\mu \in [-1,0)$. Following the steps of Algorithm~\ref{algo:s} and considering a truncation of order $1$ of the Taylor series of $\EXP{2x}$, the corresponding polynomial $H$, denoted here by $H_\mu$, is
\begin{align} \label{polyH}
H_\mu(x,\,&  \Omega) =  - \left( 1+2x \right) {\Omega}^{2}-2x \bigl(2{x}^{2}+ \left( 4\mu+1 \right) x+4{\mu}^{2}\nonumber\\
& +10\mu+4 \bigr) \Omega -2{x}^{5}+ \left( -8\mu-1 \right) {x}^{4}\nonumber\\
& -4 \left( 2\mu+1 \right)  \left( \mu-2 \right) {x}^{3}+8 \left( 2\mu+1 \right) ^{2}{x}^{2}.
\end{align}
The discriminant of $H_\mu$ with respect to the variable $\Omega$ is $\tilde{D}_\mu(x)=x^2\,D_\mu(x)$, where
 \begin{align*}
     &D_\mu(x)=\left( 64\,{\mu}^{2}+256\,\mu+128 \right) {x}^{2}+ \left( 128\,{\mu}^
{3}+576\,{\mu}^{2}\right.\\
&\left.+512\,\mu+128 \right) x+64\,{\mu}^{4}+320\,{\mu}^{3}
+656\,{\mu}^{2}+448\,\mu+96.
 \end{align*}
One easily checks that $D_\mu$ is positive only under the condition $\mu \in [-2-\sqrt{2},-2+\sqrt{2}]$ for $x>0$. In such a case, the polynomial function $H_{\mu}$ admits two real roots, denoted by $\Omega^\pm_\mu$, where $\Omega^+_\mu$ denotes the greater solution. We consider from now on $\mu \in [-1,-2+\sqrt{2}]$, which guarantees simultaneously that $q_\mu$ has a constant sign and $D_\mu$ is positive. In this case, the solution $\Omega_{\mu}^+$ is upper-bounded by 
    \begin{equation*}
        \Omega^+(x)= \frac {x}{1+2\,x} \left( -2\,{x}^{2}+3\,x+2+2\,\sqrt {4\,x+3} \right)
    \end{equation*}
    which depends only on $x$ and reaches its maximum at $x^*\approx 1.446$. Thus, ${\omega}^2=\Omega_{\mu}^+(x)<\Omega^+ (x^*)\approx 3.003<\pi^2$, i.e., $\omega<\pi$.
Finally, reasoning by contradiction, one assumes that there exists an unstable root $z_0=x+\I\,\omega \in \R^+ +\I\, \R^+$ of $\tilde\Delta$. Then, the integral representation yields 
$\int_{0}^{1} q_\mu(t)\, \EXP{-t\,z_0}
\diff t=0$ and, taking the imaginary part, we get
$\int_{0}^{1} q_\mu (t) \, \EXP{-t\,x} \sin (\omega \,t) \diff t=0.$
Now, the frequency bound $0<\omega\leq\pi$ of the previous step entails that the function $t \mapsto q_\mu (t)$ is strictly positive in $(0, 1)$, thereby contradicting the last equality. This ends the proof.
\end{proof}

It should be mentioned that Proposition~\ref{m3} can be proven by using the argument principle as done in \cite{Boussaada2020Multiplicity}. However, the proof we propose is shorter and constructive.

\begin{remark}
By Proposition~\ref{m3}, the triple root at $\lambda=\lambda_0$ is the rightmost root of \eqref{DeltaPendulum}. Thus the delay, if seen as a tuning  parameter, allows to assign the rightmost root at $\lambda=\lambda_0$ arbitrarily large (in absolute value) for small delay.
\end{remark}

\begin{remark}
A careful reading of the proof of the above result leads to the following interesting observation: the increase of the truncation order of the exponential term (see Algorithm~\ref{algo:s}) allows to enlarge the validity domain on the parameter $\mu$, and as a result this allows to enlarge the rightmost root assignability region.
\end{remark}

\section{Conclusion}
\label{sec:conclusion}

This paper discusses the spectral abscissa of linear time-invariant dynamical systems represented by delay-differential equations. It exploits the existing links between spectral values of intermediate admissible  multiplicity for a quasipolynomials and the distribution of zeros of linear combinations of Kummer confluent hypergeometric functions. It proposes a delayed control design methodology allowing the closed-loop system's solution to obey a prescribed decay rate, opening perspectives in concrete applications including, among others, vibration control (see, e.g., \cite{Boussaada2018Further}). In particular, the proposed methodology is illustrated through the stabilization problem of both the classical and the inverted pendulums.

\bibliographystyle{IEEEtran}
\bibliography{Bib}

\end{document}

%% file: Figures/invertedPendulum.tex
\begin{tikzpicture}
% Ground
\draw (-2.05, 0) -- (2.05, 0);
\foreach \x in {-2, -1.8, ..., 2}{
\draw ({\x-0.05}, -0.1) -- (\x, 0);
}

% Car
\draw[very thick, fill=gray!50] (-1.25, 0.4) rectangle (1.25, 1.4);
\node at (0, 0.9) {$M$};

% Wheels
\draw[thick, fill=gray] (-0.75, 0.2) circle[radius=0.2];
\draw[thick, fill=gray] (0.75, 0.2) circle[radius=0.2];

% Control
\draw[thick, ->] (1.75, 0.9) node[above] {$u$} -- (1.25, 0.9);

% Base of the pendulum
\draw[thick, fill=gray] (0.2, 1.4) -- (0, 1.6) -- (-0.2, 1.4) -- cycle;

% Vertical direction and angle
\draw[dashed] (0, 1.6) -- (0, 3.5);
\draw[->, thick] (0, 1.5) ++(120:0.75) arc[start angle=120, end angle=90, radius=0.75];
\draw (0, 1.5) ++(105:0.75) node[above] {$\theta$};

% Pendulum
\draw[very thick] (0, 1.5) -- ++(120:2);
\fill (0, 1.5) ++(120:2) circle[radius=0.1] node[below left] {$m$};
\end{tikzpicture}

%% file: Figures/MID_pendulum_3.pgf
%% Creator: Matplotlib, PGF backend
%%
%% To include the figure in your LaTeX document, write
%%   \input{<filename>.pgf}
%%
%% Make sure the required packages are loaded in your preamble
%%   \usepackage{pgf}
%%
%% Also ensure that all the required font packages are loaded; for instance,
%% the lmodern package is sometimes necessary when using math font.
%%   \usepackage{lmodern}
%%
%% Figures using additional raster images can only be included by \input if
%% they are in the same directory as the main LaTeX file. For loading figures
%% from other directories you can use the `import` package
%%   \usepackage{import}
%%
%% and then include the figures with
%%   \import{<path to file>}{<filename>.pgf}
%%
%% Matplotlib used the following preamble
%%   \usepackage{fontspec}
%%   \setmainfont{DejaVuSerif.ttf}[Path=\detokenize{C:/ProgramData/Anaconda3/Lib/site-packages/matplotlib/mpl-data/fonts/ttf/}]
%%   \setsansfont{DejaVuSans.ttf}[Path=\detokenize{C:/ProgramData/Anaconda3/Lib/site-packages/matplotlib/mpl-data/fonts/ttf/}]
%%   \setmonofont{DejaVuSansMono.ttf}[Path=\detokenize{C:/ProgramData/Anaconda3/Lib/site-packages/matplotlib/mpl-data/fonts/ttf/}]
%%
\begingroup%
\makeatletter%
\begin{pgfpicture}%
\pgfpathrectangle{\pgfpointorigin}{\pgfqpoint{4.000000in}{2.500000in}}%
\pgfusepath{use as bounding box, clip}%
\begin{pgfscope}%
\pgfsetbuttcap%
\pgfsetmiterjoin%
\definecolor{currentfill}{rgb}{1.000000,1.000000,1.000000}%
\pgfsetfillcolor{currentfill}%
\pgfsetlinewidth{0.000000pt}%
\definecolor{currentstroke}{rgb}{1.000000,1.000000,1.000000}%
\pgfsetstrokecolor{currentstroke}%
\pgfsetdash{}{0pt}%
\pgfpathmoveto{\pgfqpoint{0.000000in}{0.000000in}}%
\pgfpathlineto{\pgfqpoint{4.000000in}{0.000000in}}%
\pgfpathlineto{\pgfqpoint{4.000000in}{2.500000in}}%
\pgfpathlineto{\pgfqpoint{0.000000in}{2.500000in}}%
\pgfpathlineto{\pgfqpoint{0.000000in}{0.000000in}}%
\pgfpathclose%
\pgfusepath{fill}%
\end{pgfscope}%
\begin{pgfscope}%
\pgfsetbuttcap%
\pgfsetmiterjoin%
\definecolor{currentfill}{rgb}{1.000000,1.000000,1.000000}%
\pgfsetfillcolor{currentfill}%
\pgfsetlinewidth{0.000000pt}%
\definecolor{currentstroke}{rgb}{0.000000,0.000000,0.000000}%
\pgfsetstrokecolor{currentstroke}%
\pgfsetstrokeopacity{0.000000}%
\pgfsetdash{}{0pt}%
\pgfpathmoveto{\pgfqpoint{0.600000in}{0.500000in}}%
\pgfpathlineto{\pgfqpoint{3.800000in}{0.500000in}}%
\pgfpathlineto{\pgfqpoint{3.800000in}{2.250000in}}%
\pgfpathlineto{\pgfqpoint{0.600000in}{2.250000in}}%
\pgfpathlineto{\pgfqpoint{0.600000in}{0.500000in}}%
\pgfpathclose%
\pgfusepath{fill}%
\end{pgfscope}%
\begin{pgfscope}%
\pgfpathrectangle{\pgfqpoint{0.600000in}{0.500000in}}{\pgfqpoint{3.200000in}{1.750000in}}%
\pgfusepath{clip}%
\pgfsetrectcap%
\pgfsetroundjoin%
\pgfsetlinewidth{0.803000pt}%
\definecolor{currentstroke}{rgb}{0.690196,0.690196,0.690196}%
\pgfsetstrokecolor{currentstroke}%
\pgfsetdash{}{0pt}%
\pgfpathmoveto{\pgfqpoint{0.600000in}{0.500000in}}%
\pgfpathlineto{\pgfqpoint{0.600000in}{2.250000in}}%
\pgfusepath{stroke}%
\end{pgfscope}%
\begin{pgfscope}%
\pgfsetbuttcap%
\pgfsetroundjoin%
\definecolor{currentfill}{rgb}{0.000000,0.000000,0.000000}%
\pgfsetfillcolor{currentfill}%
\pgfsetlinewidth{0.803000pt}%
\definecolor{currentstroke}{rgb}{0.000000,0.000000,0.000000}%
\pgfsetstrokecolor{currentstroke}%
\pgfsetdash{}{0pt}%
\pgfsys@defobject{currentmarker}{\pgfqpoint{0.000000in}{-0.048611in}}{\pgfqpoint{0.000000in}{0.000000in}}{%
\pgfpathmoveto{\pgfqpoint{0.000000in}{0.000000in}}%
\pgfpathlineto{\pgfqpoint{0.000000in}{-0.048611in}}%
\pgfusepath{stroke,fill}%
}%
\begin{pgfscope}%
\pgfsys@transformshift{0.600000in}{0.500000in}%
\pgfsys@useobject{currentmarker}{}%
\end{pgfscope}%
\end{pgfscope}%
\begin{pgfscope}%
\definecolor{textcolor}{rgb}{0.000000,0.000000,0.000000}%
\pgfsetstrokecolor{textcolor}%
\pgfsetfillcolor{textcolor}%
\pgftext[x=0.600000in,y=0.402778in,,top]{\color{textcolor}\fontsize{10.000000}{12.000000}\selectfont 0.00}%
\end{pgfscope}%
\begin{pgfscope}%
\pgfpathrectangle{\pgfqpoint{0.600000in}{0.500000in}}{\pgfqpoint{3.200000in}{1.750000in}}%
\pgfusepath{clip}%
\pgfsetrectcap%
\pgfsetroundjoin%
\pgfsetlinewidth{0.803000pt}%
\definecolor{currentstroke}{rgb}{0.690196,0.690196,0.690196}%
\pgfsetstrokecolor{currentstroke}%
\pgfsetdash{}{0pt}%
\pgfpathmoveto{\pgfqpoint{1.133333in}{0.500000in}}%
\pgfpathlineto{\pgfqpoint{1.133333in}{2.250000in}}%
\pgfusepath{stroke}%
\end{pgfscope}%
\begin{pgfscope}%
\pgfsetbuttcap%
\pgfsetroundjoin%
\definecolor{currentfill}{rgb}{0.000000,0.000000,0.000000}%
\pgfsetfillcolor{currentfill}%
\pgfsetlinewidth{0.803000pt}%
\definecolor{currentstroke}{rgb}{0.000000,0.000000,0.000000}%
\pgfsetstrokecolor{currentstroke}%
\pgfsetdash{}{0pt}%
\pgfsys@defobject{currentmarker}{\pgfqpoint{0.000000in}{-0.048611in}}{\pgfqpoint{0.000000in}{0.000000in}}{%
\pgfpathmoveto{\pgfqpoint{0.000000in}{0.000000in}}%
\pgfpathlineto{\pgfqpoint{0.000000in}{-0.048611in}}%
\pgfusepath{stroke,fill}%
}%
\begin{pgfscope}%
\pgfsys@transformshift{1.133333in}{0.500000in}%
\pgfsys@useobject{currentmarker}{}%
\end{pgfscope}%
\end{pgfscope}%
\begin{pgfscope}%
\definecolor{textcolor}{rgb}{0.000000,0.000000,0.000000}%
\pgfsetstrokecolor{textcolor}%
\pgfsetfillcolor{textcolor}%
\pgftext[x=1.133333in,y=0.402778in,,top]{\color{textcolor}\fontsize{10.000000}{12.000000}\selectfont 0.25}%
\end{pgfscope}%
\begin{pgfscope}%
\pgfpathrectangle{\pgfqpoint{0.600000in}{0.500000in}}{\pgfqpoint{3.200000in}{1.750000in}}%
\pgfusepath{clip}%
\pgfsetrectcap%
\pgfsetroundjoin%
\pgfsetlinewidth{0.803000pt}%
\definecolor{currentstroke}{rgb}{0.690196,0.690196,0.690196}%
\pgfsetstrokecolor{currentstroke}%
\pgfsetdash{}{0pt}%
\pgfpathmoveto{\pgfqpoint{1.666667in}{0.500000in}}%
\pgfpathlineto{\pgfqpoint{1.666667in}{2.250000in}}%
\pgfusepath{stroke}%
\end{pgfscope}%
\begin{pgfscope}%
\pgfsetbuttcap%
\pgfsetroundjoin%
\definecolor{currentfill}{rgb}{0.000000,0.000000,0.000000}%
\pgfsetfillcolor{currentfill}%
\pgfsetlinewidth{0.803000pt}%
\definecolor{currentstroke}{rgb}{0.000000,0.000000,0.000000}%
\pgfsetstrokecolor{currentstroke}%
\pgfsetdash{}{0pt}%
\pgfsys@defobject{currentmarker}{\pgfqpoint{0.000000in}{-0.048611in}}{\pgfqpoint{0.000000in}{0.000000in}}{%
\pgfpathmoveto{\pgfqpoint{0.000000in}{0.000000in}}%
\pgfpathlineto{\pgfqpoint{0.000000in}{-0.048611in}}%
\pgfusepath{stroke,fill}%
}%
\begin{pgfscope}%
\pgfsys@transformshift{1.666667in}{0.500000in}%
\pgfsys@useobject{currentmarker}{}%
\end{pgfscope}%
\end{pgfscope}%
\begin{pgfscope}%
\definecolor{textcolor}{rgb}{0.000000,0.000000,0.000000}%
\pgfsetstrokecolor{textcolor}%
\pgfsetfillcolor{textcolor}%
\pgftext[x=1.666667in,y=0.402778in,,top]{\color{textcolor}\fontsize{10.000000}{12.000000}\selectfont 0.50}%
\end{pgfscope}%
\begin{pgfscope}%
\pgfpathrectangle{\pgfqpoint{0.600000in}{0.500000in}}{\pgfqpoint{3.200000in}{1.750000in}}%
\pgfusepath{clip}%
\pgfsetrectcap%
\pgfsetroundjoin%
\pgfsetlinewidth{0.803000pt}%
\definecolor{currentstroke}{rgb}{0.690196,0.690196,0.690196}%
\pgfsetstrokecolor{currentstroke}%
\pgfsetdash{}{0pt}%
\pgfpathmoveto{\pgfqpoint{2.200000in}{0.500000in}}%
\pgfpathlineto{\pgfqpoint{2.200000in}{2.250000in}}%
\pgfusepath{stroke}%
\end{pgfscope}%
\begin{pgfscope}%
\pgfsetbuttcap%
\pgfsetroundjoin%
\definecolor{currentfill}{rgb}{0.000000,0.000000,0.000000}%
\pgfsetfillcolor{currentfill}%
\pgfsetlinewidth{0.803000pt}%
\definecolor{currentstroke}{rgb}{0.000000,0.000000,0.000000}%
\pgfsetstrokecolor{currentstroke}%
\pgfsetdash{}{0pt}%
\pgfsys@defobject{currentmarker}{\pgfqpoint{0.000000in}{-0.048611in}}{\pgfqpoint{0.000000in}{0.000000in}}{%
\pgfpathmoveto{\pgfqpoint{0.000000in}{0.000000in}}%
\pgfpathlineto{\pgfqpoint{0.000000in}{-0.048611in}}%
\pgfusepath{stroke,fill}%
}%
\begin{pgfscope}%
\pgfsys@transformshift{2.200000in}{0.500000in}%
\pgfsys@useobject{currentmarker}{}%
\end{pgfscope}%
\end{pgfscope}%
\begin{pgfscope}%
\definecolor{textcolor}{rgb}{0.000000,0.000000,0.000000}%
\pgfsetstrokecolor{textcolor}%
\pgfsetfillcolor{textcolor}%
\pgftext[x=2.200000in,y=0.402778in,,top]{\color{textcolor}\fontsize{10.000000}{12.000000}\selectfont 0.75}%
\end{pgfscope}%
\begin{pgfscope}%
\pgfpathrectangle{\pgfqpoint{0.600000in}{0.500000in}}{\pgfqpoint{3.200000in}{1.750000in}}%
\pgfusepath{clip}%
\pgfsetrectcap%
\pgfsetroundjoin%
\pgfsetlinewidth{0.803000pt}%
\definecolor{currentstroke}{rgb}{0.690196,0.690196,0.690196}%
\pgfsetstrokecolor{currentstroke}%
\pgfsetdash{}{0pt}%
\pgfpathmoveto{\pgfqpoint{2.733333in}{0.500000in}}%
\pgfpathlineto{\pgfqpoint{2.733333in}{2.250000in}}%
\pgfusepath{stroke}%
\end{pgfscope}%
\begin{pgfscope}%
\pgfsetbuttcap%
\pgfsetroundjoin%
\definecolor{currentfill}{rgb}{0.000000,0.000000,0.000000}%
\pgfsetfillcolor{currentfill}%
\pgfsetlinewidth{0.803000pt}%
\definecolor{currentstroke}{rgb}{0.000000,0.000000,0.000000}%
\pgfsetstrokecolor{currentstroke}%
\pgfsetdash{}{0pt}%
\pgfsys@defobject{currentmarker}{\pgfqpoint{0.000000in}{-0.048611in}}{\pgfqpoint{0.000000in}{0.000000in}}{%
\pgfpathmoveto{\pgfqpoint{0.000000in}{0.000000in}}%
\pgfpathlineto{\pgfqpoint{0.000000in}{-0.048611in}}%
\pgfusepath{stroke,fill}%
}%
\begin{pgfscope}%
\pgfsys@transformshift{2.733333in}{0.500000in}%
\pgfsys@useobject{currentmarker}{}%
\end{pgfscope}%
\end{pgfscope}%
\begin{pgfscope}%
\definecolor{textcolor}{rgb}{0.000000,0.000000,0.000000}%
\pgfsetstrokecolor{textcolor}%
\pgfsetfillcolor{textcolor}%
\pgftext[x=2.733333in,y=0.402778in,,top]{\color{textcolor}\fontsize{10.000000}{12.000000}\selectfont 1.00}%
\end{pgfscope}%
\begin{pgfscope}%
\pgfpathrectangle{\pgfqpoint{0.600000in}{0.500000in}}{\pgfqpoint{3.200000in}{1.750000in}}%
\pgfusepath{clip}%
\pgfsetrectcap%
\pgfsetroundjoin%
\pgfsetlinewidth{0.803000pt}%
\definecolor{currentstroke}{rgb}{0.690196,0.690196,0.690196}%
\pgfsetstrokecolor{currentstroke}%
\pgfsetdash{}{0pt}%
\pgfpathmoveto{\pgfqpoint{3.266667in}{0.500000in}}%
\pgfpathlineto{\pgfqpoint{3.266667in}{2.250000in}}%
\pgfusepath{stroke}%
\end{pgfscope}%
\begin{pgfscope}%
\pgfsetbuttcap%
\pgfsetroundjoin%
\definecolor{currentfill}{rgb}{0.000000,0.000000,0.000000}%
\pgfsetfillcolor{currentfill}%
\pgfsetlinewidth{0.803000pt}%
\definecolor{currentstroke}{rgb}{0.000000,0.000000,0.000000}%
\pgfsetstrokecolor{currentstroke}%
\pgfsetdash{}{0pt}%
\pgfsys@defobject{currentmarker}{\pgfqpoint{0.000000in}{-0.048611in}}{\pgfqpoint{0.000000in}{0.000000in}}{%
\pgfpathmoveto{\pgfqpoint{0.000000in}{0.000000in}}%
\pgfpathlineto{\pgfqpoint{0.000000in}{-0.048611in}}%
\pgfusepath{stroke,fill}%
}%
\begin{pgfscope}%
\pgfsys@transformshift{3.266667in}{0.500000in}%
\pgfsys@useobject{currentmarker}{}%
\end{pgfscope}%
\end{pgfscope}%
\begin{pgfscope}%
\definecolor{textcolor}{rgb}{0.000000,0.000000,0.000000}%
\pgfsetstrokecolor{textcolor}%
\pgfsetfillcolor{textcolor}%
\pgftext[x=3.266667in,y=0.402778in,,top]{\color{textcolor}\fontsize{10.000000}{12.000000}\selectfont 1.25}%
\end{pgfscope}%
\begin{pgfscope}%
\pgfpathrectangle{\pgfqpoint{0.600000in}{0.500000in}}{\pgfqpoint{3.200000in}{1.750000in}}%
\pgfusepath{clip}%
\pgfsetrectcap%
\pgfsetroundjoin%
\pgfsetlinewidth{0.803000pt}%
\definecolor{currentstroke}{rgb}{0.690196,0.690196,0.690196}%
\pgfsetstrokecolor{currentstroke}%
\pgfsetdash{}{0pt}%
\pgfpathmoveto{\pgfqpoint{3.800000in}{0.500000in}}%
\pgfpathlineto{\pgfqpoint{3.800000in}{2.250000in}}%
\pgfusepath{stroke}%
\end{pgfscope}%
\begin{pgfscope}%
\pgfsetbuttcap%
\pgfsetroundjoin%
\definecolor{currentfill}{rgb}{0.000000,0.000000,0.000000}%
\pgfsetfillcolor{currentfill}%
\pgfsetlinewidth{0.803000pt}%
\definecolor{currentstroke}{rgb}{0.000000,0.000000,0.000000}%
\pgfsetstrokecolor{currentstroke}%
\pgfsetdash{}{0pt}%
\pgfsys@defobject{currentmarker}{\pgfqpoint{0.000000in}{-0.048611in}}{\pgfqpoint{0.000000in}{0.000000in}}{%
\pgfpathmoveto{\pgfqpoint{0.000000in}{0.000000in}}%
\pgfpathlineto{\pgfqpoint{0.000000in}{-0.048611in}}%
\pgfusepath{stroke,fill}%
}%
\begin{pgfscope}%
\pgfsys@transformshift{3.800000in}{0.500000in}%
\pgfsys@useobject{currentmarker}{}%
\end{pgfscope}%
\end{pgfscope}%
\begin{pgfscope}%
\definecolor{textcolor}{rgb}{0.000000,0.000000,0.000000}%
\pgfsetstrokecolor{textcolor}%
\pgfsetfillcolor{textcolor}%
\pgftext[x=3.800000in,y=0.402778in,,top]{\color{textcolor}\fontsize{10.000000}{12.000000}\selectfont 1.50}%
\end{pgfscope}%
\begin{pgfscope}%
\definecolor{textcolor}{rgb}{0.000000,0.000000,0.000000}%
\pgfsetstrokecolor{textcolor}%
\pgfsetfillcolor{textcolor}%
\pgftext[x=2.200000in,y=0.212809in,,top]{\color{textcolor}\fontsize{10.000000}{12.000000}\selectfont \(\displaystyle \tau\)}%
\end{pgfscope}%
\begin{pgfscope}%
\pgfpathrectangle{\pgfqpoint{0.600000in}{0.500000in}}{\pgfqpoint{3.200000in}{1.750000in}}%
\pgfusepath{clip}%
\pgfsetrectcap%
\pgfsetroundjoin%
\pgfsetlinewidth{0.803000pt}%
\definecolor{currentstroke}{rgb}{0.690196,0.690196,0.690196}%
\pgfsetstrokecolor{currentstroke}%
\pgfsetdash{}{0pt}%
\pgfpathmoveto{\pgfqpoint{0.600000in}{0.500000in}}%
\pgfpathlineto{\pgfqpoint{3.800000in}{0.500000in}}%
\pgfusepath{stroke}%
\end{pgfscope}%
\begin{pgfscope}%
\pgfsetbuttcap%
\pgfsetroundjoin%
\definecolor{currentfill}{rgb}{0.000000,0.000000,0.000000}%
\pgfsetfillcolor{currentfill}%
\pgfsetlinewidth{0.803000pt}%
\definecolor{currentstroke}{rgb}{0.000000,0.000000,0.000000}%
\pgfsetstrokecolor{currentstroke}%
\pgfsetdash{}{0pt}%
\pgfsys@defobject{currentmarker}{\pgfqpoint{-0.048611in}{0.000000in}}{\pgfqpoint{-0.000000in}{0.000000in}}{%
\pgfpathmoveto{\pgfqpoint{-0.000000in}{0.000000in}}%
\pgfpathlineto{\pgfqpoint{-0.048611in}{0.000000in}}%
\pgfusepath{stroke,fill}%
}%
\begin{pgfscope}%
\pgfsys@transformshift{0.600000in}{0.500000in}%
\pgfsys@useobject{currentmarker}{}%
\end{pgfscope}%
\end{pgfscope}%
\begin{pgfscope}%
\definecolor{textcolor}{rgb}{0.000000,0.000000,0.000000}%
\pgfsetstrokecolor{textcolor}%
\pgfsetfillcolor{textcolor}%
\pgftext[x=0.218022in, y=0.447238in, left, base]{\color{textcolor}\fontsize{10.000000}{12.000000}\selectfont \ensuremath{-}10}%
\end{pgfscope}%
\begin{pgfscope}%
\pgfpathrectangle{\pgfqpoint{0.600000in}{0.500000in}}{\pgfqpoint{3.200000in}{1.750000in}}%
\pgfusepath{clip}%
\pgfsetrectcap%
\pgfsetroundjoin%
\pgfsetlinewidth{0.803000pt}%
\definecolor{currentstroke}{rgb}{0.690196,0.690196,0.690196}%
\pgfsetstrokecolor{currentstroke}%
\pgfsetdash{}{0pt}%
\pgfpathmoveto{\pgfqpoint{0.600000in}{0.850000in}}%
\pgfpathlineto{\pgfqpoint{3.800000in}{0.850000in}}%
\pgfusepath{stroke}%
\end{pgfscope}%
\begin{pgfscope}%
\pgfsetbuttcap%
\pgfsetroundjoin%
\definecolor{currentfill}{rgb}{0.000000,0.000000,0.000000}%
\pgfsetfillcolor{currentfill}%
\pgfsetlinewidth{0.803000pt}%
\definecolor{currentstroke}{rgb}{0.000000,0.000000,0.000000}%
\pgfsetstrokecolor{currentstroke}%
\pgfsetdash{}{0pt}%
\pgfsys@defobject{currentmarker}{\pgfqpoint{-0.048611in}{0.000000in}}{\pgfqpoint{-0.000000in}{0.000000in}}{%
\pgfpathmoveto{\pgfqpoint{-0.000000in}{0.000000in}}%
\pgfpathlineto{\pgfqpoint{-0.048611in}{0.000000in}}%
\pgfusepath{stroke,fill}%
}%
\begin{pgfscope}%
\pgfsys@transformshift{0.600000in}{0.850000in}%
\pgfsys@useobject{currentmarker}{}%
\end{pgfscope}%
\end{pgfscope}%
\begin{pgfscope}%
\definecolor{textcolor}{rgb}{0.000000,0.000000,0.000000}%
\pgfsetstrokecolor{textcolor}%
\pgfsetfillcolor{textcolor}%
\pgftext[x=0.306388in, y=0.797238in, left, base]{\color{textcolor}\fontsize{10.000000}{12.000000}\selectfont \ensuremath{-}8}%
\end{pgfscope}%
\begin{pgfscope}%
\pgfpathrectangle{\pgfqpoint{0.600000in}{0.500000in}}{\pgfqpoint{3.200000in}{1.750000in}}%
\pgfusepath{clip}%
\pgfsetrectcap%
\pgfsetroundjoin%
\pgfsetlinewidth{0.803000pt}%
\definecolor{currentstroke}{rgb}{0.690196,0.690196,0.690196}%
\pgfsetstrokecolor{currentstroke}%
\pgfsetdash{}{0pt}%
\pgfpathmoveto{\pgfqpoint{0.600000in}{1.200000in}}%
\pgfpathlineto{\pgfqpoint{3.800000in}{1.200000in}}%
\pgfusepath{stroke}%
\end{pgfscope}%
\begin{pgfscope}%
\pgfsetbuttcap%
\pgfsetroundjoin%
\definecolor{currentfill}{rgb}{0.000000,0.000000,0.000000}%
\pgfsetfillcolor{currentfill}%
\pgfsetlinewidth{0.803000pt}%
\definecolor{currentstroke}{rgb}{0.000000,0.000000,0.000000}%
\pgfsetstrokecolor{currentstroke}%
\pgfsetdash{}{0pt}%
\pgfsys@defobject{currentmarker}{\pgfqpoint{-0.048611in}{0.000000in}}{\pgfqpoint{-0.000000in}{0.000000in}}{%
\pgfpathmoveto{\pgfqpoint{-0.000000in}{0.000000in}}%
\pgfpathlineto{\pgfqpoint{-0.048611in}{0.000000in}}%
\pgfusepath{stroke,fill}%
}%
\begin{pgfscope}%
\pgfsys@transformshift{0.600000in}{1.200000in}%
\pgfsys@useobject{currentmarker}{}%
\end{pgfscope}%
\end{pgfscope}%
\begin{pgfscope}%
\definecolor{textcolor}{rgb}{0.000000,0.000000,0.000000}%
\pgfsetstrokecolor{textcolor}%
\pgfsetfillcolor{textcolor}%
\pgftext[x=0.306388in, y=1.147238in, left, base]{\color{textcolor}\fontsize{10.000000}{12.000000}\selectfont \ensuremath{-}6}%
\end{pgfscope}%
\begin{pgfscope}%
\pgfpathrectangle{\pgfqpoint{0.600000in}{0.500000in}}{\pgfqpoint{3.200000in}{1.750000in}}%
\pgfusepath{clip}%
\pgfsetrectcap%
\pgfsetroundjoin%
\pgfsetlinewidth{0.803000pt}%
\definecolor{currentstroke}{rgb}{0.690196,0.690196,0.690196}%
\pgfsetstrokecolor{currentstroke}%
\pgfsetdash{}{0pt}%
\pgfpathmoveto{\pgfqpoint{0.600000in}{1.550000in}}%
\pgfpathlineto{\pgfqpoint{3.800000in}{1.550000in}}%
\pgfusepath{stroke}%
\end{pgfscope}%
\begin{pgfscope}%
\pgfsetbuttcap%
\pgfsetroundjoin%
\definecolor{currentfill}{rgb}{0.000000,0.000000,0.000000}%
\pgfsetfillcolor{currentfill}%
\pgfsetlinewidth{0.803000pt}%
\definecolor{currentstroke}{rgb}{0.000000,0.000000,0.000000}%
\pgfsetstrokecolor{currentstroke}%
\pgfsetdash{}{0pt}%
\pgfsys@defobject{currentmarker}{\pgfqpoint{-0.048611in}{0.000000in}}{\pgfqpoint{-0.000000in}{0.000000in}}{%
\pgfpathmoveto{\pgfqpoint{-0.000000in}{0.000000in}}%
\pgfpathlineto{\pgfqpoint{-0.048611in}{0.000000in}}%
\pgfusepath{stroke,fill}%
}%
\begin{pgfscope}%
\pgfsys@transformshift{0.600000in}{1.550000in}%
\pgfsys@useobject{currentmarker}{}%
\end{pgfscope}%
\end{pgfscope}%
\begin{pgfscope}%
\definecolor{textcolor}{rgb}{0.000000,0.000000,0.000000}%
\pgfsetstrokecolor{textcolor}%
\pgfsetfillcolor{textcolor}%
\pgftext[x=0.306388in, y=1.497238in, left, base]{\color{textcolor}\fontsize{10.000000}{12.000000}\selectfont \ensuremath{-}4}%
\end{pgfscope}%
\begin{pgfscope}%
\pgfpathrectangle{\pgfqpoint{0.600000in}{0.500000in}}{\pgfqpoint{3.200000in}{1.750000in}}%
\pgfusepath{clip}%
\pgfsetrectcap%
\pgfsetroundjoin%
\pgfsetlinewidth{0.803000pt}%
\definecolor{currentstroke}{rgb}{0.690196,0.690196,0.690196}%
\pgfsetstrokecolor{currentstroke}%
\pgfsetdash{}{0pt}%
\pgfpathmoveto{\pgfqpoint{0.600000in}{1.900000in}}%
\pgfpathlineto{\pgfqpoint{3.800000in}{1.900000in}}%
\pgfusepath{stroke}%
\end{pgfscope}%
\begin{pgfscope}%
\pgfsetbuttcap%
\pgfsetroundjoin%
\definecolor{currentfill}{rgb}{0.000000,0.000000,0.000000}%
\pgfsetfillcolor{currentfill}%
\pgfsetlinewidth{0.803000pt}%
\definecolor{currentstroke}{rgb}{0.000000,0.000000,0.000000}%
\pgfsetstrokecolor{currentstroke}%
\pgfsetdash{}{0pt}%
\pgfsys@defobject{currentmarker}{\pgfqpoint{-0.048611in}{0.000000in}}{\pgfqpoint{-0.000000in}{0.000000in}}{%
\pgfpathmoveto{\pgfqpoint{-0.000000in}{0.000000in}}%
\pgfpathlineto{\pgfqpoint{-0.048611in}{0.000000in}}%
\pgfusepath{stroke,fill}%
}%
\begin{pgfscope}%
\pgfsys@transformshift{0.600000in}{1.900000in}%
\pgfsys@useobject{currentmarker}{}%
\end{pgfscope}%
\end{pgfscope}%
\begin{pgfscope}%
\definecolor{textcolor}{rgb}{0.000000,0.000000,0.000000}%
\pgfsetstrokecolor{textcolor}%
\pgfsetfillcolor{textcolor}%
\pgftext[x=0.306388in, y=1.847238in, left, base]{\color{textcolor}\fontsize{10.000000}{12.000000}\selectfont \ensuremath{-}2}%
\end{pgfscope}%
\begin{pgfscope}%
\pgfpathrectangle{\pgfqpoint{0.600000in}{0.500000in}}{\pgfqpoint{3.200000in}{1.750000in}}%
\pgfusepath{clip}%
\pgfsetrectcap%
\pgfsetroundjoin%
\pgfsetlinewidth{0.803000pt}%
\definecolor{currentstroke}{rgb}{0.690196,0.690196,0.690196}%
\pgfsetstrokecolor{currentstroke}%
\pgfsetdash{}{0pt}%
\pgfpathmoveto{\pgfqpoint{0.600000in}{2.250000in}}%
\pgfpathlineto{\pgfqpoint{3.800000in}{2.250000in}}%
\pgfusepath{stroke}%
\end{pgfscope}%
\begin{pgfscope}%
\pgfsetbuttcap%
\pgfsetroundjoin%
\definecolor{currentfill}{rgb}{0.000000,0.000000,0.000000}%
\pgfsetfillcolor{currentfill}%
\pgfsetlinewidth{0.803000pt}%
\definecolor{currentstroke}{rgb}{0.000000,0.000000,0.000000}%
\pgfsetstrokecolor{currentstroke}%
\pgfsetdash{}{0pt}%
\pgfsys@defobject{currentmarker}{\pgfqpoint{-0.048611in}{0.000000in}}{\pgfqpoint{-0.000000in}{0.000000in}}{%
\pgfpathmoveto{\pgfqpoint{-0.000000in}{0.000000in}}%
\pgfpathlineto{\pgfqpoint{-0.048611in}{0.000000in}}%
\pgfusepath{stroke,fill}%
}%
\begin{pgfscope}%
\pgfsys@transformshift{0.600000in}{2.250000in}%
\pgfsys@useobject{currentmarker}{}%
\end{pgfscope}%
\end{pgfscope}%
\begin{pgfscope}%
\definecolor{textcolor}{rgb}{0.000000,0.000000,0.000000}%
\pgfsetstrokecolor{textcolor}%
\pgfsetfillcolor{textcolor}%
\pgftext[x=0.414412in, y=2.197238in, left, base]{\color{textcolor}\fontsize{10.000000}{12.000000}\selectfont 0}%
\end{pgfscope}%
\begin{pgfscope}%
\definecolor{textcolor}{rgb}{0.000000,0.000000,0.000000}%
\pgfsetstrokecolor{textcolor}%
\pgfsetfillcolor{textcolor}%
\pgftext[x=0.216000in,y=1.375000in,,bottom,rotate=90.000000]{\color{textcolor}\fontsize{10.000000}{12.000000}\selectfont \(\displaystyle \lambda_0\)}%
\end{pgfscope}%
\begin{pgfscope}%
\pgfpathrectangle{\pgfqpoint{0.600000in}{0.500000in}}{\pgfqpoint{3.200000in}{1.750000in}}%
\pgfusepath{clip}%
\pgfsetrectcap%
\pgfsetroundjoin%
\pgfsetlinewidth{1.505625pt}%
\definecolor{currentstroke}{rgb}{0.121569,0.466667,0.705882}%
\pgfsetstrokecolor{currentstroke}%
\pgfsetdash{}{0pt}%
\pgfpathmoveto{\pgfqpoint{0.724715in}{0.490000in}}%
\pgfpathlineto{\pgfqpoint{0.728000in}{0.537742in}}%
\pgfpathlineto{\pgfqpoint{0.738667in}{0.668859in}}%
\pgfpathlineto{\pgfqpoint{0.749333in}{0.781200in}}%
\pgfpathlineto{\pgfqpoint{0.760000in}{0.878521in}}%
\pgfpathlineto{\pgfqpoint{0.770667in}{0.963638in}}%
\pgfpathlineto{\pgfqpoint{0.781333in}{1.038705in}}%
\pgfpathlineto{\pgfqpoint{0.792000in}{1.105397in}}%
\pgfpathlineto{\pgfqpoint{0.813333in}{1.218679in}}%
\pgfpathlineto{\pgfqpoint{0.834667in}{1.311251in}}%
\pgfpathlineto{\pgfqpoint{0.856000in}{1.388290in}}%
\pgfpathlineto{\pgfqpoint{0.877333in}{1.453381in}}%
\pgfpathlineto{\pgfqpoint{0.898667in}{1.509084in}}%
\pgfpathlineto{\pgfqpoint{0.920000in}{1.557275in}}%
\pgfpathlineto{\pgfqpoint{0.941333in}{1.599365in}}%
\pgfpathlineto{\pgfqpoint{0.962667in}{1.636428in}}%
\pgfpathlineto{\pgfqpoint{0.984000in}{1.669303in}}%
\pgfpathlineto{\pgfqpoint{1.005333in}{1.698651in}}%
\pgfpathlineto{\pgfqpoint{1.026667in}{1.725000in}}%
\pgfpathlineto{\pgfqpoint{1.048000in}{1.748779in}}%
\pgfpathlineto{\pgfqpoint{1.080000in}{1.780378in}}%
\pgfpathlineto{\pgfqpoint{1.112000in}{1.807906in}}%
\pgfpathlineto{\pgfqpoint{1.144000in}{1.832082in}}%
\pgfpathlineto{\pgfqpoint{1.176000in}{1.853463in}}%
\pgfpathlineto{\pgfqpoint{1.208000in}{1.872490in}}%
\pgfpathlineto{\pgfqpoint{1.250667in}{1.894797in}}%
\pgfpathlineto{\pgfqpoint{1.293333in}{1.914195in}}%
\pgfpathlineto{\pgfqpoint{1.336000in}{1.931189in}}%
\pgfpathlineto{\pgfqpoint{1.389333in}{1.949641in}}%
\pgfpathlineto{\pgfqpoint{1.442667in}{1.965539in}}%
\pgfpathlineto{\pgfqpoint{1.496000in}{1.979336in}}%
\pgfpathlineto{\pgfqpoint{1.560000in}{1.993609in}}%
\pgfpathlineto{\pgfqpoint{1.634667in}{2.007688in}}%
\pgfpathlineto{\pgfqpoint{1.709333in}{2.019519in}}%
\pgfpathlineto{\pgfqpoint{1.794667in}{2.030817in}}%
\pgfpathlineto{\pgfqpoint{1.890667in}{2.041235in}}%
\pgfpathlineto{\pgfqpoint{1.997333in}{2.050523in}}%
\pgfpathlineto{\pgfqpoint{2.114667in}{2.058503in}}%
\pgfpathlineto{\pgfqpoint{2.242667in}{2.065048in}}%
\pgfpathlineto{\pgfqpoint{2.392000in}{2.070358in}}%
\pgfpathlineto{\pgfqpoint{2.552000in}{2.073722in}}%
\pgfpathlineto{\pgfqpoint{2.722667in}{2.074996in}}%
\pgfpathlineto{\pgfqpoint{2.893333in}{2.074007in}}%
\pgfpathlineto{\pgfqpoint{3.053333in}{2.070907in}}%
\pgfpathlineto{\pgfqpoint{3.192000in}{2.066174in}}%
\pgfpathlineto{\pgfqpoint{3.309333in}{2.060142in}}%
\pgfpathlineto{\pgfqpoint{3.405333in}{2.053090in}}%
\pgfpathlineto{\pgfqpoint{3.480000in}{2.045355in}}%
\pgfpathlineto{\pgfqpoint{3.533333in}{2.037546in}}%
\pgfpathlineto{\pgfqpoint{3.565333in}{2.031008in}}%
\pgfpathlineto{\pgfqpoint{3.586667in}{2.025000in}}%
\pgfpathlineto{\pgfqpoint{3.597333in}{2.020964in}}%
\pgfpathlineto{\pgfqpoint{3.608000in}{2.015312in}}%
\pgfpathlineto{\pgfqpoint{3.608000in}{2.015312in}}%
\pgfusepath{stroke}%
\end{pgfscope}%
\begin{pgfscope}%
\pgfpathrectangle{\pgfqpoint{0.600000in}{0.500000in}}{\pgfqpoint{3.200000in}{1.750000in}}%
\pgfusepath{clip}%
\pgfsetbuttcap%
\pgfsetroundjoin%
\pgfsetlinewidth{1.505625pt}%
\definecolor{currentstroke}{rgb}{1.000000,0.498039,0.054902}%
\pgfsetstrokecolor{currentstroke}%
\pgfsetdash{{7.500000pt}{1.500000pt}}{0.000000pt}%
\pgfpathmoveto{\pgfqpoint{0.724964in}{0.490000in}}%
\pgfpathlineto{\pgfqpoint{0.728000in}{0.534025in}}%
\pgfpathlineto{\pgfqpoint{0.738667in}{0.664831in}}%
\pgfpathlineto{\pgfqpoint{0.749333in}{0.776861in}}%
\pgfpathlineto{\pgfqpoint{0.760000in}{0.873871in}}%
\pgfpathlineto{\pgfqpoint{0.770667in}{0.958677in}}%
\pgfpathlineto{\pgfqpoint{0.781333in}{1.033432in}}%
\pgfpathlineto{\pgfqpoint{0.792000in}{1.099811in}}%
\pgfpathlineto{\pgfqpoint{0.813333in}{1.212468in}}%
\pgfpathlineto{\pgfqpoint{0.834667in}{1.304414in}}%
\pgfpathlineto{\pgfqpoint{0.856000in}{1.380825in}}%
\pgfpathlineto{\pgfqpoint{0.877333in}{1.445286in}}%
\pgfpathlineto{\pgfqpoint{0.898667in}{1.500357in}}%
\pgfpathlineto{\pgfqpoint{0.920000in}{1.547915in}}%
\pgfpathlineto{\pgfqpoint{0.941333in}{1.589369in}}%
\pgfpathlineto{\pgfqpoint{0.962667in}{1.625794in}}%
\pgfpathlineto{\pgfqpoint{0.984000in}{1.658028in}}%
\pgfpathlineto{\pgfqpoint{1.005333in}{1.686732in}}%
\pgfpathlineto{\pgfqpoint{1.026667in}{1.712436in}}%
\pgfpathlineto{\pgfqpoint{1.048000in}{1.735565in}}%
\pgfpathlineto{\pgfqpoint{1.080000in}{1.766184in}}%
\pgfpathlineto{\pgfqpoint{1.112000in}{1.792725in}}%
\pgfpathlineto{\pgfqpoint{1.144000in}{1.815905in}}%
\pgfpathlineto{\pgfqpoint{1.176000in}{1.836281in}}%
\pgfpathlineto{\pgfqpoint{1.208000in}{1.854293in}}%
\pgfpathlineto{\pgfqpoint{1.250667in}{1.875230in}}%
\pgfpathlineto{\pgfqpoint{1.293333in}{1.893238in}}%
\pgfpathlineto{\pgfqpoint{1.336000in}{1.908818in}}%
\pgfpathlineto{\pgfqpoint{1.389333in}{1.925469in}}%
\pgfpathlineto{\pgfqpoint{1.442667in}{1.939524in}}%
\pgfpathlineto{\pgfqpoint{1.496000in}{1.951430in}}%
\pgfpathlineto{\pgfqpoint{1.560000in}{1.963363in}}%
\pgfpathlineto{\pgfqpoint{1.624000in}{1.973152in}}%
\pgfpathlineto{\pgfqpoint{1.698667in}{1.982318in}}%
\pgfpathlineto{\pgfqpoint{1.784000in}{1.990311in}}%
\pgfpathlineto{\pgfqpoint{1.869333in}{1.996071in}}%
\pgfpathlineto{\pgfqpoint{1.965333in}{2.000255in}}%
\pgfpathlineto{\pgfqpoint{2.061333in}{2.002270in}}%
\pgfpathlineto{\pgfqpoint{2.157333in}{2.002253in}}%
\pgfpathlineto{\pgfqpoint{2.253333in}{2.000197in}}%
\pgfpathlineto{\pgfqpoint{2.349333in}{1.995918in}}%
\pgfpathlineto{\pgfqpoint{2.434667in}{1.989874in}}%
\pgfpathlineto{\pgfqpoint{2.509333in}{1.982286in}}%
\pgfpathlineto{\pgfqpoint{2.562667in}{1.974994in}}%
\pgfpathlineto{\pgfqpoint{2.605333in}{1.967485in}}%
\pgfpathlineto{\pgfqpoint{2.648000in}{1.957600in}}%
\pgfpathlineto{\pgfqpoint{2.680000in}{1.947429in}}%
\pgfpathlineto{\pgfqpoint{2.701333in}{1.938025in}}%
\pgfpathlineto{\pgfqpoint{2.712000in}{1.931730in}}%
\pgfpathlineto{\pgfqpoint{2.722667in}{1.923083in}}%
\pgfpathlineto{\pgfqpoint{2.733333in}{1.900000in}}%
\pgfpathlineto{\pgfqpoint{2.733333in}{1.900000in}}%
\pgfusepath{stroke}%
\end{pgfscope}%
\begin{pgfscope}%
\pgfpathrectangle{\pgfqpoint{0.600000in}{0.500000in}}{\pgfqpoint{3.200000in}{1.750000in}}%
\pgfusepath{clip}%
\pgfsetbuttcap%
\pgfsetroundjoin%
\pgfsetlinewidth{1.505625pt}%
\definecolor{currentstroke}{rgb}{0.172549,0.627451,0.172549}%
\pgfsetstrokecolor{currentstroke}%
\pgfsetdash{{10.500000pt}{4.500000pt}}{0.000000pt}%
\pgfpathmoveto{\pgfqpoint{0.725215in}{0.490000in}}%
\pgfpathlineto{\pgfqpoint{0.728000in}{0.530304in}}%
\pgfpathlineto{\pgfqpoint{0.738667in}{0.660798in}}%
\pgfpathlineto{\pgfqpoint{0.749333in}{0.772517in}}%
\pgfpathlineto{\pgfqpoint{0.760000in}{0.869214in}}%
\pgfpathlineto{\pgfqpoint{0.770667in}{0.953707in}}%
\pgfpathlineto{\pgfqpoint{0.781333in}{1.028149in}}%
\pgfpathlineto{\pgfqpoint{0.792000in}{1.094214in}}%
\pgfpathlineto{\pgfqpoint{0.813333in}{1.206242in}}%
\pgfpathlineto{\pgfqpoint{0.834667in}{1.297556in}}%
\pgfpathlineto{\pgfqpoint{0.856000in}{1.373333in}}%
\pgfpathlineto{\pgfqpoint{0.877333in}{1.437156in}}%
\pgfpathlineto{\pgfqpoint{0.898667in}{1.491587in}}%
\pgfpathlineto{\pgfqpoint{0.920000in}{1.538501in}}%
\pgfpathlineto{\pgfqpoint{0.941333in}{1.579307in}}%
\pgfpathlineto{\pgfqpoint{0.962667in}{1.615080in}}%
\pgfpathlineto{\pgfqpoint{0.984000in}{1.646658in}}%
\pgfpathlineto{\pgfqpoint{1.005333in}{1.674702in}}%
\pgfpathlineto{\pgfqpoint{1.026667in}{1.699740in}}%
\pgfpathlineto{\pgfqpoint{1.048000in}{1.722199in}}%
\pgfpathlineto{\pgfqpoint{1.080000in}{1.751800in}}%
\pgfpathlineto{\pgfqpoint{1.112000in}{1.777310in}}%
\pgfpathlineto{\pgfqpoint{1.144000in}{1.799444in}}%
\pgfpathlineto{\pgfqpoint{1.176000in}{1.818757in}}%
\pgfpathlineto{\pgfqpoint{1.208000in}{1.835688in}}%
\pgfpathlineto{\pgfqpoint{1.240000in}{1.850589in}}%
\pgfpathlineto{\pgfqpoint{1.282667in}{1.867776in}}%
\pgfpathlineto{\pgfqpoint{1.325333in}{1.882381in}}%
\pgfpathlineto{\pgfqpoint{1.368000in}{1.894813in}}%
\pgfpathlineto{\pgfqpoint{1.421333in}{1.907785in}}%
\pgfpathlineto{\pgfqpoint{1.474667in}{1.918349in}}%
\pgfpathlineto{\pgfqpoint{1.528000in}{1.926873in}}%
\pgfpathlineto{\pgfqpoint{1.592000in}{1.934799in}}%
\pgfpathlineto{\pgfqpoint{1.656000in}{1.940547in}}%
\pgfpathlineto{\pgfqpoint{1.730667in}{1.944833in}}%
\pgfpathlineto{\pgfqpoint{1.805333in}{1.946752in}}%
\pgfpathlineto{\pgfqpoint{1.880000in}{1.946424in}}%
\pgfpathlineto{\pgfqpoint{1.954667in}{1.943820in}}%
\pgfpathlineto{\pgfqpoint{2.018667in}{1.939622in}}%
\pgfpathlineto{\pgfqpoint{2.082667in}{1.933299in}}%
\pgfpathlineto{\pgfqpoint{2.136000in}{1.925991in}}%
\pgfpathlineto{\pgfqpoint{2.178667in}{1.918366in}}%
\pgfpathlineto{\pgfqpoint{2.221333in}{1.908500in}}%
\pgfpathlineto{\pgfqpoint{2.253333in}{1.898896in}}%
\pgfpathlineto{\pgfqpoint{2.274667in}{1.890861in}}%
\pgfpathlineto{\pgfqpoint{2.296000in}{1.880711in}}%
\pgfpathlineto{\pgfqpoint{2.306667in}{1.874372in}}%
\pgfpathlineto{\pgfqpoint{2.317333in}{1.866627in}}%
\pgfpathlineto{\pgfqpoint{2.328000in}{1.856368in}}%
\pgfpathlineto{\pgfqpoint{2.338667in}{1.838930in}}%
\pgfpathlineto{\pgfqpoint{2.338667in}{1.838930in}}%
\pgfusepath{stroke}%
\end{pgfscope}%
\begin{pgfscope}%
\pgfpathrectangle{\pgfqpoint{0.600000in}{0.500000in}}{\pgfqpoint{3.200000in}{1.750000in}}%
\pgfusepath{clip}%
\pgfsetbuttcap%
\pgfsetroundjoin%
\pgfsetlinewidth{1.505625pt}%
\definecolor{currentstroke}{rgb}{0.839216,0.152941,0.156863}%
\pgfsetstrokecolor{currentstroke}%
\pgfsetdash{{10.500000pt}{7.500000pt}}{0.000000pt}%
\pgfpathmoveto{\pgfqpoint{0.725468in}{0.490000in}}%
\pgfpathlineto{\pgfqpoint{0.728000in}{0.526580in}}%
\pgfpathlineto{\pgfqpoint{0.738667in}{0.656762in}}%
\pgfpathlineto{\pgfqpoint{0.749333in}{0.768167in}}%
\pgfpathlineto{\pgfqpoint{0.760000in}{0.864551in}}%
\pgfpathlineto{\pgfqpoint{0.770667in}{0.948729in}}%
\pgfpathlineto{\pgfqpoint{0.781333in}{1.022856in}}%
\pgfpathlineto{\pgfqpoint{0.792000in}{1.088606in}}%
\pgfpathlineto{\pgfqpoint{0.813333in}{1.200000in}}%
\pgfpathlineto{\pgfqpoint{0.834667in}{1.290677in}}%
\pgfpathlineto{\pgfqpoint{0.856000in}{1.365813in}}%
\pgfpathlineto{\pgfqpoint{0.877333in}{1.428991in}}%
\pgfpathlineto{\pgfqpoint{0.898667in}{1.482772in}}%
\pgfpathlineto{\pgfqpoint{0.920000in}{1.529032in}}%
\pgfpathlineto{\pgfqpoint{0.941333in}{1.569178in}}%
\pgfpathlineto{\pgfqpoint{0.962667in}{1.604286in}}%
\pgfpathlineto{\pgfqpoint{0.984000in}{1.635192in}}%
\pgfpathlineto{\pgfqpoint{1.005333in}{1.662556in}}%
\pgfpathlineto{\pgfqpoint{1.026667in}{1.686908in}}%
\pgfpathlineto{\pgfqpoint{1.048000in}{1.708673in}}%
\pgfpathlineto{\pgfqpoint{1.069333in}{1.728201in}}%
\pgfpathlineto{\pgfqpoint{1.101333in}{1.753915in}}%
\pgfpathlineto{\pgfqpoint{1.133333in}{1.776013in}}%
\pgfpathlineto{\pgfqpoint{1.165333in}{1.795095in}}%
\pgfpathlineto{\pgfqpoint{1.197333in}{1.811634in}}%
\pgfpathlineto{\pgfqpoint{1.229333in}{1.826003in}}%
\pgfpathlineto{\pgfqpoint{1.272000in}{1.842298in}}%
\pgfpathlineto{\pgfqpoint{1.314667in}{1.855823in}}%
\pgfpathlineto{\pgfqpoint{1.357333in}{1.867006in}}%
\pgfpathlineto{\pgfqpoint{1.400000in}{1.876180in}}%
\pgfpathlineto{\pgfqpoint{1.453333in}{1.885208in}}%
\pgfpathlineto{\pgfqpoint{1.506667in}{1.891874in}}%
\pgfpathlineto{\pgfqpoint{1.560000in}{1.896450in}}%
\pgfpathlineto{\pgfqpoint{1.624000in}{1.899439in}}%
\pgfpathlineto{\pgfqpoint{1.688000in}{1.899860in}}%
\pgfpathlineto{\pgfqpoint{1.752000in}{1.897737in}}%
\pgfpathlineto{\pgfqpoint{1.805333in}{1.893915in}}%
\pgfpathlineto{\pgfqpoint{1.858667in}{1.887983in}}%
\pgfpathlineto{\pgfqpoint{1.901333in}{1.881428in}}%
\pgfpathlineto{\pgfqpoint{1.944000in}{1.872829in}}%
\pgfpathlineto{\pgfqpoint{1.976000in}{1.864612in}}%
\pgfpathlineto{\pgfqpoint{2.008000in}{1.854273in}}%
\pgfpathlineto{\pgfqpoint{2.029333in}{1.845699in}}%
\pgfpathlineto{\pgfqpoint{2.050667in}{1.835099in}}%
\pgfpathlineto{\pgfqpoint{2.072000in}{1.821172in}}%
\pgfpathlineto{\pgfqpoint{2.082667in}{1.812016in}}%
\pgfpathlineto{\pgfqpoint{2.093333in}{1.800000in}}%
\pgfpathlineto{\pgfqpoint{2.104000in}{1.780624in}}%
\pgfpathlineto{\pgfqpoint{2.104000in}{1.780624in}}%
\pgfusepath{stroke}%
\end{pgfscope}%
\begin{pgfscope}%
\pgfpathrectangle{\pgfqpoint{0.600000in}{0.500000in}}{\pgfqpoint{3.200000in}{1.750000in}}%
\pgfusepath{clip}%
\pgfsetbuttcap%
\pgfsetroundjoin%
\pgfsetlinewidth{1.505625pt}%
\definecolor{currentstroke}{rgb}{0.580392,0.403922,0.741176}%
\pgfsetstrokecolor{currentstroke}%
\pgfsetdash{{7.500000pt}{1.500000pt}{1.500000pt}{1.500000pt}}{0.000000pt}%
\pgfpathmoveto{\pgfqpoint{0.725721in}{0.490000in}}%
\pgfpathlineto{\pgfqpoint{0.728000in}{0.522853in}}%
\pgfpathlineto{\pgfqpoint{0.738667in}{0.652721in}}%
\pgfpathlineto{\pgfqpoint{0.749333in}{0.763812in}}%
\pgfpathlineto{\pgfqpoint{0.760000in}{0.859881in}}%
\pgfpathlineto{\pgfqpoint{0.770667in}{0.943744in}}%
\pgfpathlineto{\pgfqpoint{0.781333in}{1.017554in}}%
\pgfpathlineto{\pgfqpoint{0.792000in}{1.082986in}}%
\pgfpathlineto{\pgfqpoint{0.813333in}{1.193742in}}%
\pgfpathlineto{\pgfqpoint{0.834667in}{1.283776in}}%
\pgfpathlineto{\pgfqpoint{0.856000in}{1.358265in}}%
\pgfpathlineto{\pgfqpoint{0.877333in}{1.420791in}}%
\pgfpathlineto{\pgfqpoint{0.898667in}{1.473913in}}%
\pgfpathlineto{\pgfqpoint{0.920000in}{1.519507in}}%
\pgfpathlineto{\pgfqpoint{0.941333in}{1.558980in}}%
\pgfpathlineto{\pgfqpoint{0.962667in}{1.593408in}}%
\pgfpathlineto{\pgfqpoint{0.984000in}{1.623625in}}%
\pgfpathlineto{\pgfqpoint{1.005333in}{1.650292in}}%
\pgfpathlineto{\pgfqpoint{1.026667in}{1.673936in}}%
\pgfpathlineto{\pgfqpoint{1.048000in}{1.694983in}}%
\pgfpathlineto{\pgfqpoint{1.069333in}{1.713781in}}%
\pgfpathlineto{\pgfqpoint{1.101333in}{1.738378in}}%
\pgfpathlineto{\pgfqpoint{1.133333in}{1.759327in}}%
\pgfpathlineto{\pgfqpoint{1.165333in}{1.777226in}}%
\pgfpathlineto{\pgfqpoint{1.197333in}{1.792543in}}%
\pgfpathlineto{\pgfqpoint{1.229333in}{1.805648in}}%
\pgfpathlineto{\pgfqpoint{1.261333in}{1.816835in}}%
\pgfpathlineto{\pgfqpoint{1.304000in}{1.829173in}}%
\pgfpathlineto{\pgfqpoint{1.346667in}{1.838961in}}%
\pgfpathlineto{\pgfqpoint{1.389333in}{1.846532in}}%
\pgfpathlineto{\pgfqpoint{1.432000in}{1.852135in}}%
\pgfpathlineto{\pgfqpoint{1.485333in}{1.856643in}}%
\pgfpathlineto{\pgfqpoint{1.538667in}{1.858586in}}%
\pgfpathlineto{\pgfqpoint{1.592000in}{1.858068in}}%
\pgfpathlineto{\pgfqpoint{1.645333in}{1.855053in}}%
\pgfpathlineto{\pgfqpoint{1.688000in}{1.850712in}}%
\pgfpathlineto{\pgfqpoint{1.730667in}{1.844425in}}%
\pgfpathlineto{\pgfqpoint{1.773333in}{1.835795in}}%
\pgfpathlineto{\pgfqpoint{1.805333in}{1.827371in}}%
\pgfpathlineto{\pgfqpoint{1.837333in}{1.816698in}}%
\pgfpathlineto{\pgfqpoint{1.858667in}{1.807876in}}%
\pgfpathlineto{\pgfqpoint{1.880000in}{1.797104in}}%
\pgfpathlineto{\pgfqpoint{1.901333in}{1.783380in}}%
\pgfpathlineto{\pgfqpoint{1.912000in}{1.774786in}}%
\pgfpathlineto{\pgfqpoint{1.922667in}{1.764314in}}%
\pgfpathlineto{\pgfqpoint{1.933333in}{1.750622in}}%
\pgfpathlineto{\pgfqpoint{1.944000in}{1.729027in}}%
\pgfpathlineto{\pgfqpoint{1.944000in}{1.729027in}}%
\pgfusepath{stroke}%
\end{pgfscope}%
\begin{pgfscope}%
\pgfpathrectangle{\pgfqpoint{0.600000in}{0.500000in}}{\pgfqpoint{3.200000in}{1.750000in}}%
\pgfusepath{clip}%
\pgfsetbuttcap%
\pgfsetroundjoin%
\pgfsetlinewidth{1.505625pt}%
\definecolor{currentstroke}{rgb}{0.549020,0.337255,0.294118}%
\pgfsetstrokecolor{currentstroke}%
\pgfsetdash{{7.500000pt}{1.500000pt}{1.500000pt}{1.500000pt}{1.500000pt}{1.500000pt}}{0.000000pt}%
\pgfpathmoveto{\pgfqpoint{0.725976in}{0.490000in}}%
\pgfpathlineto{\pgfqpoint{0.728000in}{0.519122in}}%
\pgfpathlineto{\pgfqpoint{0.738667in}{0.648676in}}%
\pgfpathlineto{\pgfqpoint{0.749333in}{0.759452in}}%
\pgfpathlineto{\pgfqpoint{0.760000in}{0.855204in}}%
\pgfpathlineto{\pgfqpoint{0.770667in}{0.938750in}}%
\pgfpathlineto{\pgfqpoint{0.781333in}{1.012242in}}%
\pgfpathlineto{\pgfqpoint{0.792000in}{1.077355in}}%
\pgfpathlineto{\pgfqpoint{0.813333in}{1.187468in}}%
\pgfpathlineto{\pgfqpoint{0.834667in}{1.276854in}}%
\pgfpathlineto{\pgfqpoint{0.856000in}{1.350689in}}%
\pgfpathlineto{\pgfqpoint{0.877333in}{1.412554in}}%
\pgfpathlineto{\pgfqpoint{0.898667in}{1.465007in}}%
\pgfpathlineto{\pgfqpoint{0.920000in}{1.509925in}}%
\pgfpathlineto{\pgfqpoint{0.941333in}{1.548713in}}%
\pgfpathlineto{\pgfqpoint{0.962667in}{1.582445in}}%
\pgfpathlineto{\pgfqpoint{0.984000in}{1.611956in}}%
\pgfpathlineto{\pgfqpoint{1.005333in}{1.637905in}}%
\pgfpathlineto{\pgfqpoint{1.026667in}{1.660819in}}%
\pgfpathlineto{\pgfqpoint{1.048000in}{1.681122in}}%
\pgfpathlineto{\pgfqpoint{1.069333in}{1.699161in}}%
\pgfpathlineto{\pgfqpoint{1.101333in}{1.722589in}}%
\pgfpathlineto{\pgfqpoint{1.133333in}{1.742328in}}%
\pgfpathlineto{\pgfqpoint{1.165333in}{1.758972in}}%
\pgfpathlineto{\pgfqpoint{1.197333in}{1.772981in}}%
\pgfpathlineto{\pgfqpoint{1.229333in}{1.784719in}}%
\pgfpathlineto{\pgfqpoint{1.261333in}{1.794471in}}%
\pgfpathlineto{\pgfqpoint{1.304000in}{1.804774in}}%
\pgfpathlineto{\pgfqpoint{1.346667in}{1.812361in}}%
\pgfpathlineto{\pgfqpoint{1.389333in}{1.817529in}}%
\pgfpathlineto{\pgfqpoint{1.432000in}{1.820480in}}%
\pgfpathlineto{\pgfqpoint{1.474667in}{1.821331in}}%
\pgfpathlineto{\pgfqpoint{1.517333in}{1.820117in}}%
\pgfpathlineto{\pgfqpoint{1.560000in}{1.816779in}}%
\pgfpathlineto{\pgfqpoint{1.602667in}{1.811137in}}%
\pgfpathlineto{\pgfqpoint{1.645333in}{1.802832in}}%
\pgfpathlineto{\pgfqpoint{1.677333in}{1.794465in}}%
\pgfpathlineto{\pgfqpoint{1.709333in}{1.783723in}}%
\pgfpathlineto{\pgfqpoint{1.730667in}{1.774823in}}%
\pgfpathlineto{\pgfqpoint{1.752000in}{1.764018in}}%
\pgfpathlineto{\pgfqpoint{1.773333in}{1.750492in}}%
\pgfpathlineto{\pgfqpoint{1.784000in}{1.742241in}}%
\pgfpathlineto{\pgfqpoint{1.794667in}{1.732529in}}%
\pgfpathlineto{\pgfqpoint{1.805333in}{1.720647in}}%
\pgfpathlineto{\pgfqpoint{1.816000in}{1.705027in}}%
\pgfpathlineto{\pgfqpoint{1.826667in}{1.680101in}}%
\pgfpathlineto{\pgfqpoint{1.826667in}{1.680101in}}%
\pgfusepath{stroke}%
\end{pgfscope}%
\begin{pgfscope}%
\pgfpathrectangle{\pgfqpoint{0.600000in}{0.500000in}}{\pgfqpoint{3.200000in}{1.750000in}}%
\pgfusepath{clip}%
\pgfsetbuttcap%
\pgfsetroundjoin%
\pgfsetlinewidth{1.505625pt}%
\definecolor{currentstroke}{rgb}{0.890196,0.466667,0.760784}%
\pgfsetstrokecolor{currentstroke}%
\pgfsetdash{{1.500000pt}{1.500000pt}}{0.000000pt}%
\pgfpathmoveto{\pgfqpoint{0.726232in}{0.490000in}}%
\pgfpathlineto{\pgfqpoint{0.728000in}{0.515388in}}%
\pgfpathlineto{\pgfqpoint{0.738667in}{0.644626in}}%
\pgfpathlineto{\pgfqpoint{0.749333in}{0.755086in}}%
\pgfpathlineto{\pgfqpoint{0.760000in}{0.850521in}}%
\pgfpathlineto{\pgfqpoint{0.770667in}{0.933748in}}%
\pgfpathlineto{\pgfqpoint{0.781333in}{1.006920in}}%
\pgfpathlineto{\pgfqpoint{0.792000in}{1.071711in}}%
\pgfpathlineto{\pgfqpoint{0.813333in}{1.181178in}}%
\pgfpathlineto{\pgfqpoint{0.834667in}{1.269910in}}%
\pgfpathlineto{\pgfqpoint{0.856000in}{1.343084in}}%
\pgfpathlineto{\pgfqpoint{0.877333in}{1.404280in}}%
\pgfpathlineto{\pgfqpoint{0.898667in}{1.456055in}}%
\pgfpathlineto{\pgfqpoint{0.920000in}{1.500285in}}%
\pgfpathlineto{\pgfqpoint{0.941333in}{1.538374in}}%
\pgfpathlineto{\pgfqpoint{0.962667in}{1.571395in}}%
\pgfpathlineto{\pgfqpoint{0.984000in}{1.600182in}}%
\pgfpathlineto{\pgfqpoint{1.005333in}{1.625393in}}%
\pgfpathlineto{\pgfqpoint{1.026667in}{1.647552in}}%
\pgfpathlineto{\pgfqpoint{1.048000in}{1.667083in}}%
\pgfpathlineto{\pgfqpoint{1.069333in}{1.684333in}}%
\pgfpathlineto{\pgfqpoint{1.101333in}{1.706537in}}%
\pgfpathlineto{\pgfqpoint{1.133333in}{1.725000in}}%
\pgfpathlineto{\pgfqpoint{1.165333in}{1.740307in}}%
\pgfpathlineto{\pgfqpoint{1.197333in}{1.752911in}}%
\pgfpathlineto{\pgfqpoint{1.229333in}{1.763163in}}%
\pgfpathlineto{\pgfqpoint{1.261333in}{1.771338in}}%
\pgfpathlineto{\pgfqpoint{1.293333in}{1.777648in}}%
\pgfpathlineto{\pgfqpoint{1.336000in}{1.783433in}}%
\pgfpathlineto{\pgfqpoint{1.378667in}{1.786448in}}%
\pgfpathlineto{\pgfqpoint{1.421333in}{1.786833in}}%
\pgfpathlineto{\pgfqpoint{1.464000in}{1.784605in}}%
\pgfpathlineto{\pgfqpoint{1.506667in}{1.779635in}}%
\pgfpathlineto{\pgfqpoint{1.538667in}{1.773918in}}%
\pgfpathlineto{\pgfqpoint{1.570667in}{1.766221in}}%
\pgfpathlineto{\pgfqpoint{1.602667in}{1.756117in}}%
\pgfpathlineto{\pgfqpoint{1.624000in}{1.747697in}}%
\pgfpathlineto{\pgfqpoint{1.645333in}{1.737524in}}%
\pgfpathlineto{\pgfqpoint{1.666667in}{1.725000in}}%
\pgfpathlineto{\pgfqpoint{1.688000in}{1.709023in}}%
\pgfpathlineto{\pgfqpoint{1.698667in}{1.699078in}}%
\pgfpathlineto{\pgfqpoint{1.709333in}{1.687111in}}%
\pgfpathlineto{\pgfqpoint{1.720000in}{1.671918in}}%
\pgfpathlineto{\pgfqpoint{1.730667in}{1.650237in}}%
\pgfpathlineto{\pgfqpoint{1.730667in}{1.650237in}}%
\pgfusepath{stroke}%
\end{pgfscope}%
\begin{pgfscope}%
\pgfsetrectcap%
\pgfsetmiterjoin%
\pgfsetlinewidth{0.803000pt}%
\definecolor{currentstroke}{rgb}{0.000000,0.000000,0.000000}%
\pgfsetstrokecolor{currentstroke}%
\pgfsetdash{}{0pt}%
\pgfpathmoveto{\pgfqpoint{0.600000in}{0.500000in}}%
\pgfpathlineto{\pgfqpoint{0.600000in}{2.250000in}}%
\pgfusepath{stroke}%
\end{pgfscope}%
\begin{pgfscope}%
\pgfsetrectcap%
\pgfsetmiterjoin%
\pgfsetlinewidth{0.803000pt}%
\definecolor{currentstroke}{rgb}{0.000000,0.000000,0.000000}%
\pgfsetstrokecolor{currentstroke}%
\pgfsetdash{}{0pt}%
\pgfpathmoveto{\pgfqpoint{3.800000in}{0.500000in}}%
\pgfpathlineto{\pgfqpoint{3.800000in}{2.250000in}}%
\pgfusepath{stroke}%
\end{pgfscope}%
\begin{pgfscope}%
\pgfsetrectcap%
\pgfsetmiterjoin%
\pgfsetlinewidth{0.803000pt}%
\definecolor{currentstroke}{rgb}{0.000000,0.000000,0.000000}%
\pgfsetstrokecolor{currentstroke}%
\pgfsetdash{}{0pt}%
\pgfpathmoveto{\pgfqpoint{0.600000in}{0.500000in}}%
\pgfpathlineto{\pgfqpoint{3.800000in}{0.500000in}}%
\pgfusepath{stroke}%
\end{pgfscope}%
\begin{pgfscope}%
\pgfsetrectcap%
\pgfsetmiterjoin%
\pgfsetlinewidth{0.803000pt}%
\definecolor{currentstroke}{rgb}{0.000000,0.000000,0.000000}%
\pgfsetstrokecolor{currentstroke}%
\pgfsetdash{}{0pt}%
\pgfpathmoveto{\pgfqpoint{0.600000in}{2.250000in}}%
\pgfpathlineto{\pgfqpoint{3.800000in}{2.250000in}}%
\pgfusepath{stroke}%
\end{pgfscope}%
\begin{pgfscope}%
\definecolor{textcolor}{rgb}{0.000000,0.000000,0.000000}%
\pgfsetstrokecolor{textcolor}%
\pgfsetfillcolor{textcolor}%
\pgftext[x=2.200000in,y=2.333333in,,base]{\color{textcolor}\fontsize{12.000000}{14.400000}\selectfont Behavior of the triple root at \(\displaystyle \lambda_0\)}%
\end{pgfscope}%
\begin{pgfscope}%
\pgfsetbuttcap%
\pgfsetmiterjoin%
\definecolor{currentfill}{rgb}{1.000000,1.000000,1.000000}%
\pgfsetfillcolor{currentfill}%
\pgfsetfillopacity{0.800000}%
\pgfsetlinewidth{1.003750pt}%
\definecolor{currentstroke}{rgb}{0.800000,0.800000,0.800000}%
\pgfsetstrokecolor{currentstroke}%
\pgfsetstrokeopacity{0.800000}%
\pgfsetdash{}{0pt}%
\pgfpathmoveto{\pgfqpoint{1.580924in}{0.569444in}}%
\pgfpathlineto{\pgfqpoint{3.702778in}{0.569444in}}%
\pgfpathquadraticcurveto{\pgfqpoint{3.730556in}{0.569444in}}{\pgfqpoint{3.730556in}{0.597222in}}%
\pgfpathlineto{\pgfqpoint{3.730556in}{1.398762in}}%
\pgfpathquadraticcurveto{\pgfqpoint{3.730556in}{1.426540in}}{\pgfqpoint{3.702778in}{1.426540in}}%
\pgfpathlineto{\pgfqpoint{1.580924in}{1.426540in}}%
\pgfpathquadraticcurveto{\pgfqpoint{1.553147in}{1.426540in}}{\pgfqpoint{1.553147in}{1.398762in}}%
\pgfpathlineto{\pgfqpoint{1.553147in}{0.597222in}}%
\pgfpathquadraticcurveto{\pgfqpoint{1.553147in}{0.569444in}}{\pgfqpoint{1.580924in}{0.569444in}}%
\pgfpathlineto{\pgfqpoint{1.580924in}{0.569444in}}%
\pgfpathclose%
\pgfusepath{stroke,fill}%
\end{pgfscope}%
\begin{pgfscope}%
\pgfsetrectcap%
\pgfsetroundjoin%
\pgfsetlinewidth{1.505625pt}%
\definecolor{currentstroke}{rgb}{0.121569,0.466667,0.705882}%
\pgfsetstrokecolor{currentstroke}%
\pgfsetdash{}{0pt}%
\pgfpathmoveto{\pgfqpoint{1.608702in}{1.314072in}}%
\pgfpathlineto{\pgfqpoint{1.747591in}{1.314072in}}%
\pgfpathlineto{\pgfqpoint{1.886480in}{1.314072in}}%
\pgfusepath{stroke}%
\end{pgfscope}%
\begin{pgfscope}%
\definecolor{textcolor}{rgb}{0.000000,0.000000,0.000000}%
\pgfsetstrokecolor{textcolor}%
\pgfsetfillcolor{textcolor}%
\pgftext[x=1.997591in,y=1.265461in,left,base]{\color{textcolor}\fontsize{10.000000}{12.000000}\selectfont g/L = 1}%
\end{pgfscope}%
\begin{pgfscope}%
\pgfsetbuttcap%
\pgfsetroundjoin%
\pgfsetlinewidth{1.505625pt}%
\definecolor{currentstroke}{rgb}{1.000000,0.498039,0.054902}%
\pgfsetstrokecolor{currentstroke}%
\pgfsetdash{{7.500000pt}{1.500000pt}}{0.000000pt}%
\pgfpathmoveto{\pgfqpoint{1.608702in}{1.110215in}}%
\pgfpathlineto{\pgfqpoint{1.747591in}{1.110215in}}%
\pgfpathlineto{\pgfqpoint{1.886480in}{1.110215in}}%
\pgfusepath{stroke}%
\end{pgfscope}%
\begin{pgfscope}%
\definecolor{textcolor}{rgb}{0.000000,0.000000,0.000000}%
\pgfsetstrokecolor{textcolor}%
\pgfsetfillcolor{textcolor}%
\pgftext[x=1.997591in,y=1.061604in,left,base]{\color{textcolor}\fontsize{10.000000}{12.000000}\selectfont g/L = 2}%
\end{pgfscope}%
\begin{pgfscope}%
\pgfsetbuttcap%
\pgfsetroundjoin%
\pgfsetlinewidth{1.505625pt}%
\definecolor{currentstroke}{rgb}{0.172549,0.627451,0.172549}%
\pgfsetstrokecolor{currentstroke}%
\pgfsetdash{{10.500000pt}{4.500000pt}}{0.000000pt}%
\pgfpathmoveto{\pgfqpoint{1.608702in}{0.906358in}}%
\pgfpathlineto{\pgfqpoint{1.747591in}{0.906358in}}%
\pgfpathlineto{\pgfqpoint{1.886480in}{0.906358in}}%
\pgfusepath{stroke}%
\end{pgfscope}%
\begin{pgfscope}%
\definecolor{textcolor}{rgb}{0.000000,0.000000,0.000000}%
\pgfsetstrokecolor{textcolor}%
\pgfsetfillcolor{textcolor}%
\pgftext[x=1.997591in,y=0.857747in,left,base]{\color{textcolor}\fontsize{10.000000}{12.000000}\selectfont g/L = 3}%
\end{pgfscope}%
\begin{pgfscope}%
\pgfsetbuttcap%
\pgfsetroundjoin%
\pgfsetlinewidth{1.505625pt}%
\definecolor{currentstroke}{rgb}{0.839216,0.152941,0.156863}%
\pgfsetstrokecolor{currentstroke}%
\pgfsetdash{{10.500000pt}{7.500000pt}}{0.000000pt}%
\pgfpathmoveto{\pgfqpoint{1.608702in}{0.702501in}}%
\pgfpathlineto{\pgfqpoint{1.747591in}{0.702501in}}%
\pgfpathlineto{\pgfqpoint{1.886480in}{0.702501in}}%
\pgfusepath{stroke}%
\end{pgfscope}%
\begin{pgfscope}%
\definecolor{textcolor}{rgb}{0.000000,0.000000,0.000000}%
\pgfsetstrokecolor{textcolor}%
\pgfsetfillcolor{textcolor}%
\pgftext[x=1.997591in,y=0.653890in,left,base]{\color{textcolor}\fontsize{10.000000}{12.000000}\selectfont g/L = 4}%
\end{pgfscope}%
\begin{pgfscope}%
\pgfsetbuttcap%
\pgfsetroundjoin%
\pgfsetlinewidth{1.505625pt}%
\definecolor{currentstroke}{rgb}{0.580392,0.403922,0.741176}%
\pgfsetstrokecolor{currentstroke}%
\pgfsetdash{{7.500000pt}{1.500000pt}{1.500000pt}{1.500000pt}}{0.000000pt}%
\pgfpathmoveto{\pgfqpoint{2.780740in}{1.314072in}}%
\pgfpathlineto{\pgfqpoint{2.919629in}{1.314072in}}%
\pgfpathlineto{\pgfqpoint{3.058518in}{1.314072in}}%
\pgfusepath{stroke}%
\end{pgfscope}%
\begin{pgfscope}%
\definecolor{textcolor}{rgb}{0.000000,0.000000,0.000000}%
\pgfsetstrokecolor{textcolor}%
\pgfsetfillcolor{textcolor}%
\pgftext[x=3.169629in,y=1.265461in,left,base]{\color{textcolor}\fontsize{10.000000}{12.000000}\selectfont g/L = 5}%
\end{pgfscope}%
\begin{pgfscope}%
\pgfsetbuttcap%
\pgfsetroundjoin%
\pgfsetlinewidth{1.505625pt}%
\definecolor{currentstroke}{rgb}{0.549020,0.337255,0.294118}%
\pgfsetstrokecolor{currentstroke}%
\pgfsetdash{{7.500000pt}{1.500000pt}{1.500000pt}{1.500000pt}{1.500000pt}{1.500000pt}}{0.000000pt}%
\pgfpathmoveto{\pgfqpoint{2.780740in}{1.110215in}}%
\pgfpathlineto{\pgfqpoint{2.919629in}{1.110215in}}%
\pgfpathlineto{\pgfqpoint{3.058518in}{1.110215in}}%
\pgfusepath{stroke}%
\end{pgfscope}%
\begin{pgfscope}%
\definecolor{textcolor}{rgb}{0.000000,0.000000,0.000000}%
\pgfsetstrokecolor{textcolor}%
\pgfsetfillcolor{textcolor}%
\pgftext[x=3.169629in,y=1.061604in,left,base]{\color{textcolor}\fontsize{10.000000}{12.000000}\selectfont g/L = 6}%
\end{pgfscope}%
\begin{pgfscope}%
\pgfsetbuttcap%
\pgfsetroundjoin%
\pgfsetlinewidth{1.505625pt}%
\definecolor{currentstroke}{rgb}{0.890196,0.466667,0.760784}%
\pgfsetstrokecolor{currentstroke}%
\pgfsetdash{{1.500000pt}{1.500000pt}}{0.000000pt}%
\pgfpathmoveto{\pgfqpoint{2.780740in}{0.906358in}}%
\pgfpathlineto{\pgfqpoint{2.919629in}{0.906358in}}%
\pgfpathlineto{\pgfqpoint{3.058518in}{0.906358in}}%
\pgfusepath{stroke}%
\end{pgfscope}%
\begin{pgfscope}%
\definecolor{textcolor}{rgb}{0.000000,0.000000,0.000000}%
\pgfsetstrokecolor{textcolor}%
\pgfsetfillcolor{textcolor}%
\pgftext[x=3.169629in,y=0.857747in,left,base]{\color{textcolor}\fontsize{10.000000}{12.000000}\selectfont g/L = 7}%
\end{pgfscope}%
\end{pgfpicture}%
\makeatother%
\endgroup%